\documentclass[reqno,11pt]{amsart}

\usepackage{amsmath}               % AmSLaTeX
\usepackage{amsthm}                
\usepackage{latexsym}
\usepackage[ansinew]{inputenc}
\usepackage{graphicx}
\usepackage{hyperref}
\usepackage{amssymb}
\usepackage{xspace, xcolor}
\usepackage{amscd}
\usepackage{ulem}
\usepackage[all]{xy}
\newcommand{\pr}[2]{\ensuremath{\langle {#1},{#2}\rangle}}

\newtheorem{theorem}{Theorem}[section]
\newtheorem{definition}[theorem]{Definition}
\newtheorem{lemma}[theorem]{Lemma}
\newtheorem{corollary}[theorem]{Corollary}
\newtheorem{proposition}[theorem]{Proposition}

\newtheorem{remark}[theorem]{Remark}
\newtheorem{notation}[theorem]{Notation}

\newcommand{\HOM}{\mathrm{HOM}}
\newcommand{\END}{\mathrm{END}}
\newcommand{\End}{\mathrm{End}}
\newcommand{\Hom}{\mathrm{Hom}}
\newcommand{\gr}{\mathrm{-gr}}
\newcommand{\grd}{\mathrm{gr-}}
\newcommand{\GR}{\mathrm{-GR}}
\newcommand{\GRD}{\mathrm{GR-}}
\newcommand{\bimod}[2]{#1\mathrm{-}#2}

\newcommand{\flecha}{\rightarrow}
\newcommand{\flechag}{\longrightarrow}

\newcommand{\tq}{\mid}
 
\newcommand{\flechadecor}[1]{\xrightarrow{~#1~}}

\begin{document}
	
\title{Graded Equivalence for Graded Idempotent Rings}

	\author{Mikhailo Dokuchaev$^a$}
	\address{$^a$Instituto de Matem\'atica e Estat\'istica, Universidade de S\~ao Paulo,  Rua do Mat\~ao, 1010, S\~ao Paulo, SP,  CEP: 05508--090, Brazil,  E-mail: \texttt{dokucha@gmail.com}}
	%\email{dokucha@gmail.com}
% 	\thanks{Corresponding author: Mikhailo Dokuchaev (\texttt{dokucha@gmail.com})}
	
	\author{Juan Jacobo Sim\'on$^c$}
	\address{$^c$Departamento de Matem\'{a}ticas, Universidad de Murcia, 30071 Murcia, Spain,  E-mail: \texttt{jsimon@um.es}}
	%\email{jsimon@um.es} 


\maketitle

\begin{abstract}
In this paper, we extend the study of graded equivalences to the case of general idempotent graded  rings. We prove that the existence of a graded equivalence between two categories of graded torsion-free unital modules may be characterized by the existence of a Morita context with surjective trace maps. As an  application of our results we relate certain lattices of graded submodules and graded ideals of  graded equivalent garded rings and  give some properties invariant under graded equivalences.
\end{abstract}
% 
% \begin{keywords}
% graded ring \sep idempotent ring\sep graded module\sep graded equivalence.
% %quadrupole exciton \sep polariton \sep \WGM \sep \BEC
% \end{keywords}
% 
% \maketitle

\section{Introduction}

The classical  Morita Theory has  been addressed for the case of graded rings by dealing with two kinds of graded-equivalences: one, the graded-equivalence, as defined in \cite{GordonGreen}  and \cite{MeniniNast}, and a stronger one, the graded Morita equivalence, given in \cite{Boisen}. As it has ocurred in the classical theory (see \cite{Abrams}, \cite{AnhMarki}), the study of graded equivalences has been carried out to the case of graded rings  with local units  
in \cite{Haefner94}  (graded-equivalence) and  \cite{Haefner95} (graded Morita  equivalence). Besides rings with local units, the clasical Morita Theory  was extended to idempotent rings in \cite{Nobusawa} and \cite{GS}. In particular, it was proved in \cite{Nobusawa} that, given an appropriate  Morita  context between idempotent torsion-free rings, there exists an equivalence between  the categories of unital torsion-free modules. This was extended in \cite{GS} for general idempotent rings, proving also that a category equivalence, as above,  implies the existence of an appropriate Morita context.   

In \cite{AbDoEx2024}, besides the new concept of a strong equivalence for graded rings, the graded equivalence from  
\cite{GordonGreen}  and \cite{MeniniNast} was considered for general idempotent  graded rings  in the language of Morita contexts. Thus, it is natural to study the relation between Morita contexts and categories of modules for general idempotent graded rings, a task acomplished in the present article. 

After the preliminary Section~\ref{backgroud}, we prove in Section~\ref{Sec Morita Equiv} that  the graded functors, which give  a graded  equivalence between  the categories of unital torsion-free graded modules, can be realized by means of appropriate HOM functors (see Proposition~\ref{F y G iso a los HOM}). In the next Section~\ref{ SecMorita contexts}, given two idempotent torsion-free graded  rings $A$ and $B$  and a graded equivalence between the categories of their unital torsion-free graded modules, we produce a non-degenerate graded Morita context between $A$ and $B$ with surjective trace maps (see Theorem~\ref{Car Equiv}). The main result of Section~\ref{sec:FromMoritaContextToEquiv} is
Theorem~\ref{contextos a equivalencias final}, which states that if there exists a graded Morita context between idempotent torsion-free graded  rings $A$ and $B$  with torsion-free graded  bimodules and surjective trace maps, then the categories of  unital torsion-free graded modules over $A$ and $B$ are graded equivalent, and the functors which give the graded equivalence are described in terms of $\otimes $ and ${\rm HOM}.$    

With respect to the graded equivalence studied for general idempotent graded rings in \cite{AbDoEx2024}, we relax  the torsion-free restriction on the rings in Section~\ref{sec:Generalization}, proving in
Theorem~\ref{teo:Generalization} that, given a graded Morita context with surjective trace maps between idempotent not necessarily torsion-free graded rings, their categories of unital torsion-free graded modules are graded equivalent, and the equivalence functors again involve $\otimes$ and ${\rm HOM}.$ Observe that, in view of Lemma~\ref{el HOM siempre es torsionfree}, it can be seen from this, that  the category of unital torsion-free graded modules is the largest category of modules over general idempotent graded rings, for which graded equivalence can be considered in terms of HOM-functors.

In the final Section~\ref{sec:Applications} we give some applications for idempotent torsion-free graded rings, producing isomorphisms between  appropriate lattices of graded submodules and lattices of the unital torsion-free graded two-sided ideals, as well as  listing some properties of graded modules and rings  invariant under  graded equivalences.

\section{Background.}\label{backgroud}

We take all standard definitions related to graded rings from \cite{GS} and \cite{methods}. All rings considered in this paper are associative  group  graded rings, but we do not assume that they have  a unity element. 

 Let $\Gamma$ be a multiplicative group, with identity element $e\in \Gamma$,  which will be fixed  \underline{in all what follows}. Let, furthermore,  $A$ be a  $\Gamma$-graded ring
and let $M$ be a  $\Gamma$-graded left $A$-module. We say that $A$ is {\it idempotent} if $A^2=A$ and $M$ is said to be {\it unital} if $AM=M$.  As usual we define the {\it left $A$-torsion part} of $M$ as $t_A(M)=\{m\in M\tq Am=0\}$. It is well-known that $t_A(M)$ is a graded unital left $A$-submodule of $M$. We shall say that $M$ is {\it torsion-free} if $t_A(M)=0$; that is, for any $m\in M$,  the equality $Am=0$ implies $m=0$. There are analogous definitions for graded right $A$-modules.  For rings, we have to distinguish between left and right torsion. So we denote the left torsion of $A$ by $l(A)$ and the right one by  $r(A).$

 \underline{Throughout this paper} by a   {\it graded ring} (respectively,  {\it module}) we shall mean a ring (respectively, module) graded by the group $\Gamma.$ \underline{In all what follows} $A$ will denote an idempotent  graded ring, which occasionally will be assumed to be left and  right torsion-free.  We denote by $A\gr$ (resp. $\grd A$) the category of all unital and torsion-free left (resp. right)  graded  $A$-modules.  We shall say that {\it a  graded ring is torsion-free} if it is left and right torsion-free; in this case,  $A\in A\gr $ and $A \in \grd A .$ 
 
 As we mentioned in the Introduction, we will study the graded equivalences between categories of unital and torsion-free modules; however, these categories, as it is known, are not closed, for example, under quotients or tensor products. Furthermore, some Morita contexts that we will consider will not be formed by torsion-free modules, although they will always be unital. Therefore, we will   consider the category $A\gr$ (resp. $\grd A$) as a full subcategory of the category $A\GR$ (resp. $\GRD A$) of left (resp. right) unital $A$-modules.

We recall that, for any  $M\in A\GR$ and $\sigma\in \Gamma$, the {\it $\sigma$-suspension} of $M,$  which we denote by $M(\sigma)$, is the graded left $A$-module such that $M(\sigma)_\tau=M_{\tau\sigma}$.

Let  $M,N\in A\GR$. The {\it morphisms in the category} $A\GR$ are defined as
 \[\Hom_{A\GR}(M,N)=\{f\in \Hom_A(M,N)\tq  f(M_\sigma)\subseteq N_\sigma  \text{  for all  }\sigma\in \Gamma\}.\]  
 The morphisms in the subcategory $A\gr$ will be denoted by $\Hom_{A\gr}(M,N)$, for $M,N\in A\gr$. Note that, in this case, $\Hom_{A\gr}(M,N)=\Hom_{A\GR}(M,N)$.
 
 On the other hand, we recall the definition of a  graded homomorphism of a given degree from \cite[Section 2.4]{methods}. Let  $M,N\in A\GR$. A homomorphism $f\in \Hom_A(M,N)$ is called a {\it graded homomorphism of degree} $\sigma\in \Gamma$ if  $f(M_\tau)\subseteq N_{\tau\sigma}$, for all $\tau\in \Gamma$. Clearly, for a fixed $\sigma\in \Gamma$ the set of all graded homomorphisms of degree $\sigma$  from $M$ to $N$ form an abelian group under the usual sum, which  we denote  by $\HOM_A(M,N)_\sigma$. The set of all graded $A$-homomorphisms  (of all degrees) is denoted by $\HOM_A(M,N)$.
 The following equality is easy to see
\begin{equation}\label{eq:HOM_direct_sum_decomposition}
\HOM_A(M,N)=\oplus_{\sigma\in \Gamma} \HOM_A(M,N)_\sigma, \text{  a direct sum of abelian groups}.
\end{equation}
    Obviously, $\HOM_A(M,N)_e=\Hom_{A\GR}(M,N)$; moreover,
 \begin{equation}\label{eq:HOM_ditect_sum_components}
\HOM_A(M,N)_\sigma= \Hom_{A\GR}(M,N(\sigma))=\Hom_{A\GR}(M(\sigma^{-1}),N).
 \end{equation} 
 
   Symmetrically, for  $M,\,N\in \GRD A$,  a {\it graded homomorphism of degree} $\sigma\in \Gamma$ is defined by requiring  $f(M_\tau)\subseteq N_{\sigma\tau}$, for all $\tau\in \Gamma$. For the set of all graded homomorphisms of graded right $A$-modules  we shall use the same notation $\HOM_A(M,N)$. By a {\it graded homomorphism of graded $A$-bimodules} we shall mean an $A$-bimodule homomorphism, which is a graded homomorphism of left (equivalently, right) graded $A$-modules  of degree $e\in \Gamma$.
 
  We often write homomorphisms acting (multiplicatively) on the opposite side to the scalars; for example, for $M,N\in A\gr$,   $f\in \HOM_A(M,N)$ (or $f\in \Hom_{A\GR}(M,N)$) and $m\in M$, we write $f(m)=mf$, and we shall use both notations.
 
 As usual, we denote the ring of  endomorphisms (respectively, graded endomorphisms) of any  $M\in A\GR$,  by  $\End(_{A\GR} M)$ (resp. $\END(_AM)$, see \cite{methods}). Following our notation above, we write endomorphisms acting as opposite scalars; that is, for any $\alpha\in \End(_{A\GR} M)$ or $\END(_A M)$ and $m\in M$, we write $\alpha(m)=m\alpha$.
 
For $M,N\in A\GR$, we define the {\it graded trace} of $M$ in $N$ as $$Tr_N(M)=\sum\{Im(f)\tq f\in\Hom_{A\GR}(M(\tau),N),\;\tau\in \Gamma\}.$$ We recall that $Tr_N(M)$ is a graded unital submodule of $N$ and the  graded trace of $M$ in the ring $A$ is a graded two-sided ideal in $A$.
 
 \begin{remark}\label{la traza de HOM}
 	It is immediately seen that, for $M,N\in A\GR$ we have that $$Tr_N(M)=\sum\{Im(f)\tq f\in \HOM_A(M,N)\}.$$
 \end{remark}
 
 As in the case of graded rings with identity (see \cite{Boisen}), we define a {\it graded Morita context} as a  sextuple $(A,B,P,Q,\mu,\nu),$ in which $A$ and $B$ are idempotent   graded rings, $_AP_B$ is a graded $\bimod{A}{B}$-bimodule and $_BN_A$ is a graded $\bimod{B}{A}$-bimodule; both of them  unital at both sides. The maps $\mu:P\otimes _B Q\flecha A$ and $\nu:Q\otimes_A P\flecha B$ (called {\it trace maps}) are graded bimodule homomorphisms, satisfying the following associativity conditions. For $m,m'\in M$ and $n,n'\in N$, setting $\mu(m\otimes n)=\pr{m}{n}\in A$ and $\nu(n\otimes m)=[n,m]\in B$ we have that
 \begin{align}\label{condiciones asoc Contextos}
 m'[n,m]=\pr{m'}{n}m \quad\text{and}\quad n'\pr{m}{n}=[n',m]n.
 \end{align} 

 \begin{remark}\label{torsion en tor}
It is worth to mention that, even when $_AP_B$ and $_BQ_A$ are unital torsion-free bimodules, the graded bimodules $P\otimes_B Q$ and $Q\otimes_A P$ are unital as bimodules, but not necessarily torsion-free.
\end{remark}

 \begin{definition}\label{generador unital}
 Let $_A P$ and $_A M$ be unital graded modules. We say that $P$ generates $M$ if for any unital graded module $_A N$ and $0\neq f\in \Hom_{A\GR}(M,N)$ there exist $\sigma\in \Gamma$ and $g\in \Hom_{A\GR}(P(\sigma),M)$ such that $f\circ g\neq 0$, where $P(\sigma)$ is the $\sigma$-suspension.
\end{definition}

 Following \cite{hazrat}, we define a generator of $A\gr$, as follows.

 \begin{definition}\label{generador en A-gr}
 A module $P\in A\gr$ is a graded generator of the category $A\gr$ if it generates every module from $A\gr$.
\end{definition}

The following result should be known, but  we  prove it for the reader's convenience. 

\begin{proposition}\label{unital genera unital}
 Let $_A P$ and $_A M$ be unital graded modules. Then $P$ generates $M$ if and only if $Tr_M(P)=M$.
\end{proposition}

\begin{proof}
	First, assume that $P$ generates $M$ and set $K=M/Tr_M(P)$. If $K$ is a torsion $A$-module then $AK=0$, and hence $AM\subseteq Tr_M(P)$. As $M=AM$ we get $M=Tr_M(P)$. Otherwise, let $\eta: M\flecha K/t_A(K) \neq 0$ be the composition of the canonical maps $M \to M/Tr_M(P) \to K / t_A(K)$, which is non-zero.  But then $\eta \circ f=0$ for all $\sigma\in \Gamma$ and $f\in \Hom_{A\GR}(P(\sigma),M)$, which contradicts Definition~\ref{generador unital}.
	
   Conversely, consider $N\in A\GR$ and any $0\neq f\in \Hom_{A\GR}(M,N)$. Then, there exists $m\in M$ such that $f(m)\neq 0$. By hypothesis, there is a linear combination $m=\sum g_i(p_i)$ and then $0\neq f(m)=\sum (f\circ g_i)(p_i)$ which means that $f\circ g_i\neq 0$  for some $g_i \in  \HOM_A(P,M)$.
\end{proof}

\begin{corollary}\label{generador y traza}
 A module $P\in A\gr$ is a graded generator of the category $A\gr$ if and only if $Tr_M(P)=M$, for all $M\in A\gr$.
\end{corollary}

{\begin{remark}\label{traza sobe el anillo hace generador}
Note that,  when $A$ is idempotent and torsionfree and since all our $A$-modules are unital, it follows that $A$ is  a graded generator of the category $A\gr .$  Consequently,  in order to see  that $P\in A\gr$ is  a graded generator of the category $A\gr ,$   it is enough to check that  $Tr_A(P)=A$.
\end{remark}

 Given an arbitrary  $\sigma\in\Gamma$, we shall use   the $\sigma$-suspension functor $T_\sigma$, which is the functor $A\gr \to A\gr ,$  defined on the objects by  $T_\sigma(M)=M(\sigma)$ and which acts trivially on the morphisms.

\begin{definition} (see \cite{hazrat},\cite{Haefner94})
 Let $A, B$ idempotent graded rings. An aditive functor $F:A\gr\flechag B\gr$ is said to be a graded functor if it commutes with the $\sigma$-suspension  functor, for all $\sigma\in\Gamma$; that is,  $T_\sigma\circ F =F\circ T_\sigma$.  If, in addition, $F$ determines an equivalence of the categories $A\gr $ and $ B\gr ,$ then we say that $F:A\gr\flechag B\gr$ is a graded category equivalence. In this case we say that $A$ and $B$ are graded equivalent.
\end{definition}

 Obviously,  the  equality  $T_\sigma\circ F =F\circ T_\sigma$ means that $F(M(\sigma))= F(M)(\sigma)$ for any $M\in A\gr $.

\section{Graded equivalence}\label{Sec Morita Equiv}

 Let us recall the graded left $B$-module structure on $\HOM_A(M,N)$ for any  unital graded bimodule $_A M_B$ and   an arbitrary graded module $_A N$. Given $f\in \HOM_A(M,N)$, $m\in M$ and $b\in B$ we have,  by definition,  that $(bf)(m)=f(mb)$.  Then, in view of \eqref{eq:HOM_direct_sum_decomposition},  it is easily seen that $\HOM_A(M,N)$ is a graded left $B$-module.

We begin with the next easy fact.

\begin{lemma}\label{el HOM siempre es torsionfree}
Let $A, B$ be idempotent  graded rings, $_A M_B$ a unital graded bimodule and  $_A N$ a graded module. Then 
$\HOM_A(M,N)$ is torsion-free as a left $B$-module.
\end{lemma}
\begin{proof}
Let  $f\in \HOM_A(M,N)$. If $f\in t_B(\HOM_A(M,N))$ then 
$0=(Bf)(M)=f(MB)=f(M) ,$ so that  $f=0$.
\end{proof}

  We  proceed with a result which is essential for our purpose. Viewing $_A A_A$ as a bimodule, using the above defined  action, we obtain a graded left $A$-module structure on $\HOM_A(A,M)$.

\begin{proposition}\label{isomorfismo clave}
 Let $A$ be an idempotent  graded ring.	For any $M\in A\gr$ there is a natural graded isomorphism in $A\gr$
	\[\chi_M:M \flecha A\cdot \HOM_A(A,M),\]
	 such that  $\chi_M(m)(a)=am$ for all $a\in A$.
\end{proposition}
\begin{proof}
 It directly follows from the above defined left $A$-action  on $\HOM_A(A,M)$ and Lemma~\ref{el HOM siempre es torsionfree} that $A\cdot \HOM_A(A,M)\in A\gr .$ Obviously,  the above defined map $\chi_M: M\flechag \HOM_A(A,M)$ is additive. Taking 
	 $a\in I$ and $m\in M,$ we see that
	\[(a\chi_M(m))(x)=\chi_M(m)(xa)=(xa)m=x(am)=\chi_M(am)(x),\]
	showing that $\chi_M$ is  left $A$-linear. 
	
 Let us check that $\chi_M(M_\sigma)\subseteq \HOM_A(A,M)_\sigma$. Taking $m\in M_\sigma$, we have 	for all $\tau\in \Gamma$ and $a\in A_\tau$  that $\chi_M(m)(a)=am\in M_{\tau\sigma},$ so that  $\chi_M(m)\in \HOM_A(A,M)_\sigma, $ proving that  $\chi_M$ is a morphism in $A\gr .$
 
  We proceed by showing that $\chi_M$ is an isomorphism. To see the injectivity, suppose that  $\chi_M(m)=0$, for some $m\in M.$  Then, for all $a\in A$ we have that $am=0$, so that $Am=0$ and, since $M\in A\gr$, it follows that $m=0$. 
  
  Let us see that $Im\,\chi_M = A\cdot \HOM_A(A,M).$ Take any $m\in M$ and write it as $m=\sum_ia_im_i$. Then $\chi_M(m)=\chi_M(\sum_ia_im_i)=\sum_ia_i\chi_M(m_i)\in A\cdot \HOM_A(A,M)$, showing that $Im\,\chi_M \subseteq A\cdot \HOM_A(A,M)$. For the converse inclusion  let $a\in A$ and $f\in \HOM_A(A,M)$. By definition, for all $x\in A$, $(af)(x)=f(xa)=x\cdot f(a)$. So that $(af)=\chi_M(f(a))\in Im\,\varphi$.	Hence  $\chi_M:M\flechag A\cdot \HOM_A(A,M)$ is an isomorphism in $A\gr$.\\
  
  It remains to check that $\chi_{-}$ is natural.  Consider $M,N\in A\gr$ and a morphism in $A\gr$, $f:M\flecha N$.
  As usual (see \cite[Section 2.4]{methods}), we set $\HOM_A(A,f)=\cdot f$; that is, for any $x\in \HOM_A(A,M)$, we have $x\cdot f=f\circ x \in \HOM_A(A,N)$, the composition. We denote by $\cdot f_|$ the restriction of $\cdot f$ to $A\cdot \HOM_A(A,M)$.
	
For any $a\in A$ and $x\in \HOM_A(A,M)$, we have that $(ax)f_|=(ax)f=a(x \cdot f)\in A\cdot \HOM_A(A,N)$; so that $Im\,(\cdot f_|)\subseteq A\cdot \HOM_A(A,N)$. Now, we have to see that the following diagram is commutative.
\begin{center}
  \begin{picture}(230,85)(0,0)
  \put(5,10){$A\cdot \HOM_A(A,M)$}
  \put(60,70){$M$}
    \put(140,10){$A\cdot \HOM_A(A,N)$}
  \put(170,70){$N$}
  \put(80,72){\vector(1,0){70}}
   \put(95,12){\vector(1,0){35}}
    \put(68,60){\vector(0,-1){35}}
    \put(170,60){\vector(0,-1){35}}
   \put(51,45){$\chi_M$}
      \put(153,45){$\chi_N$}
  \put(110,80){$f$}
  \put(105,20){$\cdot f_|$}
 \end{picture}
 \end{center} 
 
 To see it, let $m\in M$ and $a\in A$. Then,
\[(f_| \circ \chi_M(m))(a)=f(am)=af(m)=\chi_N(f(m))(a)=\left((\chi_N\circ f)(m)\right)(a).  \]  This completes our proof.
\end{proof}

\begin{lemma}\label{HOM natural} Let $A, B$ be idempotent  torsion-free graded rings, and let $F:A\gr\flechag B\gr$ be a graded category equivalence. For $M,N\in A\gr$ we have a group isomorphism
\[\HOM_A(M,N) \cong \HOM_B(F(M),F(N))\] which is graded of degree $e$ and natural  in both variables.
\end{lemma}
\begin{proof}
For any $\sigma\in \Gamma$
 \begin{eqnarray*}
  \HOM_A(M,N)_\sigma&=&\Hom_{A\gr}(M(\sigma^{-1}),N) \cong\Hom_{B\gr}(F(M(\sigma^{-1})),F(N))\cong\\
  &\cong& \Hom_{B\gr}(F(M)(\sigma^{-1}),F(N))=\HOM_B(F(M),F(N))_\sigma ,
 \end{eqnarray*}
 giving an isomorphism. The naturality is easily verified.
\end{proof}

Let $A, B$ be idempotent  torsion-free graded rings. It is known that there is a graded ring homomorphism $\alpha: A\flecha \END(_AA)$ via the right multiplication by the elements of $A$; in fact, $\alpha(A)$ is a right ideal of $\END(_AA)$. On the other hand, by the previous lemma, there is a graded ring isomorphism
\[\END(_AA)=\HOM_A(A,A)\cong_{gr} \HOM_B(F(A),F(A))=\END(_B F(A)).\]

As usual, we denote this isomorphism by $F: \END(_AA)\flecha \END(_B F(A))$. We write $\bar\alpha=F \circ \alpha$; that is, for any $a\in A$, we have $\bar\alpha(a)=F(\alpha(a))$.

Now, writing endomorphisms acting as opposite scalars, we may endow $F(A)$ with a structure of a graded right $\END(_A A)$-module; that is, for any $x\in F(A)$ and $f \in \END(_A A)$ we set $x \cdot f = x F(f).$ In particular,
$F(A)$ is a graded right  
$A$-module such that for any $x\in F(A)$ we  have  $xa=x\bar\alpha(a)$.  Analogously, we obtain  $\beta: B\flecha \END(_B B)$ and $G :  \END(_B B)\flecha \END(_A G(B))$; so that $G(B)$ has a structure of a graded right $\END(_B B)$-module, in particular, that of a graded right $B$-module.

Let us see these  actions with more details.
\begin{remark}\label{Q por I}
We keep the above notation; that is, for any $x\in F(A)$ and $a\in A$, we have $xa=x\bar\alpha(a)$. We consider the exact sequence $A \flechadecor{\alpha(a)} A\alpha(a)\flecha 0$. Now, apply the functor $F$ to obtain the exact sequence $F(A) \flechadecor{F(\alpha(a))} F(A\alpha(a))\flecha 0$.

This means that $F(A)F(\alpha(a))= F(A\alpha(a))$.  Since $\bar\alpha(a)=F(\alpha(a))$, we have that $F(A)\bar\alpha(a)= F(A\alpha(a))$.

Analogously, for any $y\in G(B)$ and $b\in B$, we have $yb=y\bar\beta(b)$,  and since  $\bar\beta(b)=G(\beta(b))$, we obtain $G(B)\bar\beta(b)= G(B\beta(b))$.

In general, for any $x\in \END(_A A)$ and $y\in \END(_B B)$ we have that $F(A)F(x)=F(Ax)$ and $G(B)G(y)=G(By)$. We shall simplify this notation by writing   $F(A) \cdot x =  F(A)F(x)= F(Ax)$ and $G(B) \cdot y = G(B)G(y)=G(By).$ 
\end{remark}

As a consequence of this remark, we may see that $F(A)$ is right $A$-torsion-free and $G(B)$ is right $B$-torsion-free. Let us  check the first assertion. Since $A$ is  right $A$-torsion-free, then $\bigcap_{a\in A} \ker(\alpha(a))=0$. By the remark above, one has that $F(\ker(\alpha(a)))=\ker(F(\alpha(a)))$. Set $L=\bigcap_{a\in A} \ker(F(\alpha(a)))$. Suppose that $L\neq 0$ and consider, for every $a\in A$ the sequence $0\flecha L\flecha \ker(F(\alpha(a)))=F(\ker(\alpha(a)))$. Then $G(L)\neq 0$ and we have an exact sequence $0\flecha G(L)\flecha \ker(\alpha(a))$, for all $a\in A,$ which is impossible.

The following result shows us that, in fact, $F(A)\in \grd A$ and $G(B)\in \grd B$.

\begin{lemma}\label{PJ es P} 
In the setting described above, define $P=G(B)$ and $Q=F(A)$ with the right  actions  via change of rings: $xa=x\bar\alpha(a)$ and $yb=y\bar\beta(b)$, for $x\in Q$, $a\in A$, $y\in P$ and $b\in B$. Then $P=PB$ and $Q=QA$.

Moreover, if   $J\subseteq \END(_B B)$ is a graded two-sided ideal such that  $P \cdot J=P$ then  $\beta(B)\subseteq J$. Analogously, if  $I\subseteq \END(_A A)$ is a graded two-sided ideal such that   $Q \cdot I=Q$ then $\alpha(A)\subseteq I$.
\end{lemma}
\begin{proof}
With the notation in Remark~\ref{Q por I}, since $A$ and $B$ are idempotent graded rings, we have that $A=A\alpha(A)$ and $B=B\beta(B)$.  By the same remark, we obtain
\[PB=G(B)B=G(B)\bar\beta(B)=G(B\beta(B))=G(B)=P\]
and, analogously, $QA=Q$.

 For the second part take a two-sided ideal $J \subseteq \END(_B B)$ such that  $P \cdot J=P$.  We have, by Remark~\ref{Q por I}, that
\[G(B)=P=P \cdot J=G(B)G(J)=G(BJ).\]
By applying $F$ we get $B=BJ$. We want to see that $\beta(B)\subseteq J$. Take any $b\in B$, so it may be writen as $b=\sum_{i=1}^n b_il_i=\sum_{i=1}^n l_i(b_i)$, where $b_i\in B$ and $l_i\in J$. Consider $\sum_{i=1}^n \beta(b_i)l_i$ as a sum of elements in $\END(_B B)$ and take any $x\in B$. Then $x\sum_{i=1}^n \beta(b_i)l_i= \sum_{i=1}^n x\beta(b_i)l_i=\sum_{i=1}^nl_i(xb_i)=x\sum_{i=1}^nl_i(b_i)=xb= x \beta(b) $; so that $\beta(b)=\sum_{i=1}^n \beta(b_i)l_i\in J$. The other case is completely analogous.
\end{proof}

Sumarizing, we have the following result.

\begin{proposition}\label{P y Q son bimod sin torsion}
Let $A, B$ be idempotent  torsion-free graded rings, and let $F:A\gr\flechag B\gr$ and  $G:B\gr\flechag A\gr$ be inverse graded category equivalences. We set, as above $P=G(B)$ and $Q=F(A)$. Then $_AP_B$ and $_BQ_A$ have structures of graded bimodules and, moreover,
 \begin{itemize}
 \item[(i)]  $_AP\in A\gr$ and $P_B\in \grd B$.
  \item[(ii)] $_BQ\in B\gr$ and $Q_A\in \grd A$.
 \end{itemize}
\end{proposition}
%\red{\begin{proof}
% Immediate from the above results.
%\end{proof}}

An interesting consequence of Proposition~\ref{P y Q son bimod sin torsion}  is that, for any $N\in B\gr$  we have a structure of a graded left $A$-module on   $\HOM_B(F(A),N)$ via $(af)(x)=f(xa)$, with $a\in A$ and $x\in F(A),$ and there is  a similar  structure of a left $B$-module on $\HOM_A(G(B),M)$. This will be usefull to prove the following result.

\begin{lemma}\label{F conmuta con A y el otro}
In the setting of Proposition~\ref{P y Q son bimod sin torsion},  the following equalities hold:
\begin{itemize}
\item[(i)] For any $M\in A\gr$,  we have that
$$F\left(A\cdot \HOM_A(A,M)\right)=A\cdot {F\left(\HOM_A(A,M)\right)}=A\cdot \HOM_B(Q,F(M)).$$
\item[(ii)] For any $N\in B\gr$,  we have that
$$G\left(B\cdot \HOM_B(B,N)\right)=B\cdot { G\left(\HOM_B(B,N)\right)}=B\cdot \HOM_A(P,G(N)).$$
\end{itemize}
\end{lemma}
\begin{proof}
(i) Consider $a\in A$ and $f\in   \HOM_A(A,M)$. We know that we may write $af$ as the composition $f\circ \alpha(a):A\flecha A\flecha M$. By applying $F$ we get 
$$F(af)=F(f\circ \alpha(a))=F(f)\circ F(\alpha(a)):F(A)\flecha F(A)\flecha F(M);$$
so that, for any $x\in F(A)$ we have
$$x \cdot F(af)=x\cdot\left(F(\alpha(a))F(f)\right)=xF(\alpha(a))\cdot F(f)=(xa)F(f)=x\left(aF(f)\right).$$
% $x\mapsto x F(\alpha(a))\mapsto x F(\alpha(a))F(f)$. As we have seen in  Remark~\ref{Q por I} we have that $xa=x F(\alpha(a))=F(\alpha(a))(x)$; so that $ x F(\alpha(a))F(f)=F(f)(xa)=\left(aF(f)\right)(x)$. 
This means that $F(af)=aF(f),$  showing  (i).

The proof of (ii) is completely analogous.
\end{proof}

\begin{remark}\label{rem:f_commuta_con_a} Note that in the proof of Lemma~\ref{F conmuta con A y el otro}  we showed that $F(af)=aF(f),$ for any $a\in A$ and $f\in   \HOM_A(A,M)$. Similarly, $G(bg)=bG(g)$ for all $b\in B$ and $g \in    \HOM_B(B,N)$.   
\end{remark}

 Observe that Lemma~\ref{F conmuta con A y el otro} implies that the isomorphism  
$$F: \HOM_A(A,M) \to \HOM_B(F(A),F(M))$$ restricts to an isomorphism 
$$ A\cdot \HOM_A(A,M) \to 
A\cdot \HOM_B(F(A),F(M))$$ for any 
 $M\in A\gr$.

 Note also that by Proposition~\ref{isomorfismo clave}, for any $N\in B\gr ,$   there is a natural graded isomorphism of graded left $A$-modules, $$G(N) \cong A \cdot \HOM_A(A,G(N)). $$ Applying $F$  to $A \cdot \HOM_A(A,G(N))$ we obtain by Lemma~\ref{HOM natural} a natural isomorphism 
$$G(N) \cong F(A \cdot \HOM_A(A,G(N))) ,$$ which by Lemma~\ref{F conmuta con A y el otro} results in a   graded isomorphism of graded left $A$-modules, $$	G(N)\cong_{gr-nat} A\cdot \HOM_B(F(A),N)=A\cdot \HOM_B(Q,N),$$ that is natural in $N.$ Analogoulsy, 
there is a natural graded isomorphism of graded left $B$-modules,
$$	F(M)\cong_{gr-nat} B\cdot\HOM_A(G(B),M)=B\cdot\HOM_A(P,M).$$
for any  $M\in A\gr$.

As a direct consequence, we have that following result.

\begin{proposition}\label{F y G iso a los HOM}
Let $A, B$ be idempotent and torsion-free graded rings, and let $F:A\gr\flechag B\gr$ and  $G:B\gr\flechag A\gr$ be inverse graded category equivalences. Consider the setting described in paragraphs above with $P=G(B)$ and $Q=F(A)$. Then there are natural graded isomorphisms
\begin{eqnarray*}
	G&\cong_{gr-nat}& A\cdot \HOM_B(F(A),-)=A\cdot \HOM_B(Q,-),\\ \nonumber
	F&\cong_{gr-nat}& B\cdot\HOM_A(G(B),-)=B\cdot\HOM_A(P,-).
\end{eqnarray*}
\end{proposition}

\section{From graded category equivalences to graded Morita contexts.}\label{ SecMorita contexts}

Let $A, B$ be idempotent  torsion-free graded rings. We keep all notation from last section. Throughout this section, we fix natural graded isomorphisms $\eta:G\circ F\flecha 1_{A\gr}$ 
and $\gamma :F\circ G\flecha 1_{B\gr}$.
 For any $M,N\in A\gr$ the graded isomorphism $\HOM_A(G(F(M)),G(F(N))) \cong \HOM_A(M,N)$ 
induced by $\eta $ is obtained by 
\begin{equation}\label{eq:eta_iso} FG(g) \mapsto \eta \circ FG(g)\circ \eta ^{-1} =g,\end{equation} $g \in \HOM_A(M,N).$ With a slight abuse of notation, denote  by 
$G^{-1}$ the inverse isomorphism of $G : \END (_B B)\to \END (_A P) .$  Then $FG (G^{-1}f )= F(f)$ for all $f \in \END (_A P)$ and, applying \eqref{eq:eta_iso} to $g = G^{-1}(f),$ we obtain 
\begin{equation}\label{eq:F_Gorro}
  G ^{-1}(f) = \eta \circ F(f) \circ \eta ^{-1}.
\end{equation}  Write $\hat{F} : =\eta \circ F(-) \circ \eta ^{-1}.$ 

%We will write the same for $\eta'$.  
From the results above, we define a natural graded isomorphism of graded left $A$-modules, and, respectively, a natural graded isomorphism of graded left $B$-modules
\begin{align}\label{def del rho y psi}
  \rho:P\flecha A\cdot \HOM_B(Q,B)\quad\text{and}\quad \psi:Q\flecha B\cdot \HOM_A(P,A)
\end{align}
as follows.

By Proposition~\ref{isomorfismo clave} we have a graded  natural isomorphism $\chi_P:P\flecha A\cdot \HOM_A(A,P)$ and by  Lemmas~\ref{HOM natural} and \ref{F conmuta con A y el otro}, a graded isomorphism
$$\bar F:A\cdot \HOM_A(A,P)
\flechadecor{F} A\cdot \HOM_B(F(A),F(P))
\overset{\gamma \circ - }{\cong} A\cdot \HOM_B(Q,B);$$ 
that is, for $p\in P$ and $q\in Q$, $\rho(p)(q)=\bar F(\chi_P(p))(q)$. Analogously,  for $p\in P$ and $q\in Q$, $\psi_q(p)=\bar G(\chi_Q(q))(p).$ 
%\red{ where $\chi_Q:Q\flecha B\cdot \HOM_B(B,Q)$.} To see that $\rho $ is a left $A$-module map, take $a\in A, f\in \HOM_A(A,P)$ and $q \in Q.$ Then, using Remark~\ref{rem:f_commuta_con_a}, we see that 
\begin{align*}\rho (ap )(q)&= \bar{F}(\chi_Q (ap))(q)= \bar{F}(a\chi_Q (p))(q)= \gamma (F(a\chi_Q (p))(q) )\\&\overset{Remark\, \ref{rem:f_commuta_con_a}}{=} \gamma (a F(\chi_Q (p))(q) )= \gamma ( F(\chi_Q (p))(qa) ) = a\gamma ( F(\chi_Q (p))(q) )= a\rho (p )(q).
\end{align*}  A similar computation shows that $\psi$ is a left $B$-module map.

We want to see that, in fact, $\rho$ and $\psi$ are isomorphisms of \textit{bimodules}.  First, let us recall the structure of a right $B$-module of $ A\cdot \HOM_B(Q,B)$: for $f\in  A\cdot \HOM_B(Q,B)$, $q\in Q$ and $b \in B$ we have $(fb)(q)=f(q)b$. Clearly, $A\cdot \HOM_B(Q,B)$ is an $\bimod{A}{B}$-bimodule. Similarly, the right $A$-module structure on $B\cdot \HOM_A(P,A)$ is given by setting   $(ga)(p)=g(p)a$ for $g\in B\cdot \HOM_A(P,A),$ $p\in P$ and $a\in A,$ transforming  $B\cdot \HOM_A(P,A)$ into a $\bimod{B}{A}$-bimodule.

\begin{lemma}\label{lema del rho y el psi}
Let $A, B$ be idempotent and torsion-free graded rings, and let $F:A\gr\flechag B\gr$ and  $G:B\gr\flechag A\gr$ be inverse graded category equivalences. Consider the setting described in Section~\ref{Sec Morita Equiv}, with $P=G(B)$ and $Q=F(A)$ and the homomorphisms defined  in \eqref{def del rho y psi}
\[ \rho:P\flecha A\cdot \HOM_B(Q,B)\quad\text{and}\quad \psi:Q\flecha B\cdot \HOM_A(P,A).\]

Then $\rho$ is an $\bimod{A}{B}$-bimodule graded isomorphism while $\psi$ is a $\bimod{B}{A}$-bimodule graded isomorphism.
\end{lemma}
\begin{proof}
 We shall  prove that  $\psi(qa)=\psi(q)a$, for all $q\in Q$ and $a\in A$. The fact that $\rho(pb)=\rho(p)b$, for all $p\in P$ and $b\in B$ can be verified in a completely analogous way.

First, let us recall that, by definition, for any $q\in Q$ and $a\in A$, we have that $qa=\bar\alpha(a)(q)$, where  $\bar\alpha(a) =F(\alpha(a))$. On the other hand, for any  $f\in B\cdot \HOM_A(P,A)$ and $a\in A$, we have, for any $x\in P$, that $(fa)(x)=f(x)a$ and since $f(x)\in A$ then $(fa)(x)=\alpha(a)(f(x))$; so that $fa=\alpha(a)\circ f$.

 Now,  we know by Proposition~\ref{isomorfismo clave} that the isomorphism  $\chi_Q:Q\flecha B\cdot \HOM_B(B,Q)$ involved in the definition of $\psi$, is  natural, which implies that, for any $q\in Q$ we have 

 $$\bar\alpha(a) \circ\chi_Q(q) =\chi_Q(\bar\alpha(a)(q))=\chi_Q(qa).$$ 
 Applying $\bar G$ to this equality we get
 $$\bar G(\bar\alpha(a)\circ\chi_Q(q))=\bar G(\chi_Q(qa))=\psi(qa).$$ 
On the other hand,

 \begin{align*} 
& \bar G(\bar\alpha(a)\circ\chi_Q(q))=
\eta \circ G(\bar\alpha(a) \circ \chi_Q(q))= 
\eta \circ G(F(\alpha(a))) \circ G(\chi_Q(q))\\
&= \alpha(a) \circ \eta  \circ G(\chi_Q(q))
= \alpha(a) \circ  \bar{G}(\chi_Q(q))
=  \alpha(a)\circ \psi(q)=\psi(q)a.
\end{align*} 
\end{proof}

In order to construct  a graded Morita context, we begin with the trace maps. 
%{\color{teal} We set $\hat F:\END(_AP)\flechadecor{F}\END(_BF(P))\overset{{\eta}'}{\cong}\END(_BB).$   }

Consider the maps
\begin{eqnarray}\label{mu y nu primas}
\mu':P \times Q\flechag A&\text{such that}& \mu'(p,q)=(\psi(q))(p),\\ \nonumber
\nu': Q\times P\flechag \END(_B B)&\text{such that}& \nu'(q,p)= \hat F(\psi(q)(-)\cdot p),
\end{eqnarray} where $\hat{F}$ is defined after \eqref{eq:F_Gorro}.

Note that $\psi(q)(-)\cdot p\in \END(_A P)$; that is, for $x\in P$ we have $x\mapsto \psi(q)(x)\cdot p \in P$.

\begin{lemma}
In the above setting, the maps in (\ref{mu y nu primas}) are $B$- and $A$-balanced, respectively.
\end{lemma}
\begin{proof}
First, for $b\in B$, $p\in P$ and $q\in Q$ 
we know that $\psi(q)(pb)=(b\psi(q))(p)$ by  the definition of the left $B$-action on the homomorphisms in  $\HOM_A(P,A)$. Having in mind that 
$\psi$ is  left ${B}$-linear (see Lemma~\ref{lema del rho y el psi}), we have, by equality~\eqref{mu y nu primas}, that
\[ \mu'(pb, q)=\psi(q)(pb)=(b\psi(q))(p)= \psi(bq)(p)=\mu'(p,bq).\]

Now  take $x,p\in P$, $q\in Q$ and $a\in A$.  Then by the right $A$-action on $\HOM_A(P,A)$ and Lemma~\ref{lema del rho y el psi} we see that 
\[\psi(qa)(x)p =(\psi(q)a)(x)p=\psi(q)(x)ap,  \] which implies that  
$$ \psi(qa)(-)\cdot p = \psi(q)(-)\cdot (ap)$$
 in $\END(_A P).$ Then applying $\hat F$ we conclude that $\nu'(qa,p)= \nu'(q,ap),$ as desired.
\end{proof}

 Thus, we may define the following group homomorphisms given  on tensor products
\begin{eqnarray*}%\label{los mu y nu tensoriales}
	\mu:P \otimes_B Q\flechag A&\text{such that}& \mu(p\otimes q)=\psi(q)(p)\label{mu}.\\ \nonumber
	\nu: Q\otimes_A P\flechag \END(_B B)&\text{such that}& \nu(q\otimes p)= \hat F(\psi(q)(-)\cdot p).
\end{eqnarray*}

As it is usual, we denote the trace maps, for $p\in P$ and $q\in Q$, as 
\begin{equation}\label{tracemaps}
\pr{p}{q}=\mu(p\otimes q)=\psi(q)(p)\quad\text{and}\quad [q,p]=\nu(q,p)=\hat F(\psi(q)(-)\cdot p).
\end{equation}

Now, take $x,p\in P$, $q\in Q$.   We want to compute $x \cdot [q,p]$. To do this, note that $ \hat F(\psi(q)(-)\cdot p)\in \END_B(B)$ acts on $P$ via the right multiplication defined in Remark~\ref{Q por I}; that is, for any $x\in P$ and $g\in \END(_B B)$, we have $x\cdot g=x\cdot G(g)=G(g)(x)$. Then, using \eqref{eq:F_Gorro}, we obtain 
\begin{equation}\label{regla de multiplicar x[q,p]}
 x\cdot [q,p]=x\cdot \nu(q,p)=x\cdot G(\hat F(\psi(q)(-)\cdot p))=x\cdot \psi(q)(-)\cdot p=\psi(q)(x)\cdot p.
\end{equation}

Note that \eqref {regla de multiplicar x[q,p]} implies the following  associativity of the trace maps in  \eqref{condiciones asoc Contextos}
\begin{align}\label{Asociatividad facil}
  x[q,p]=\psi(q)(x)\cdot p=\pr{x}{q}\cdot p.
\end{align}

From now on, for the sake of simplicity, we identify $B$ and $\beta(B)$. We recall that if $R$ is a graded ring with identity and $K$ is a graded right (or left) ideal, then $RK$ (or $KR$) is a graded two-sided ideal.

In order to show that   $(A,B,P,Q,\mu,\nu)$ is a graded Morita context, we shall prove  the rest of the properties, beginning with those involving  $\mu$.

To see that  $\mu$ is $A$-bilinear take $a\in A$, $p\in P$ and $q\in Q.$ Then   by \eqref{tracemaps} we have that
 \begin{align*}
 \pr{ap}{q}=&\psi(q)(ap)\overset{\eqref{def del rho y psi}}{=}a\cdot \psi(q)(p)=a\pr{p}{q}.\\
\pr{p}{qa}=&\psi(qa)(p)\overset{Lemma~\ref{lema del rho y el psi}}{=}(\psi(q)a)(p) =\psi(q)(p)a=\pr{p}{q}\cdot a,
\end{align*} showing the $A$-bilinearity.

 Next, consider $\{\gamma,\delta,\sigma,\tau\}\subseteq\Gamma$, $a\in A_\gamma$, $a'\in A_\delta$, $p\in P_\sigma$ and $q\in Q_\tau$. We know that $a\pr{p}{q}a'=\pr{ap}{qa'}$. 
We have that $qa'\in Q_{\tau\delta}$, so that $\psi(qa')\in B\cdot \HOM_A(P,A)_{\tau\delta}$ and then, by definition, $\pr{ap}{qa'} = \psi(qa')(ap)\in A_{\gamma\sigma\tau\delta}$, proving that $\mu$ is a graded bimodule homomorphism (of degree $e$).

Now we check that $\mu$ is onto. Take an arbitrary  $a\in A$. Since $_A P$ is a generator, by Corollary~\ref{generador y traza} and Remark~\ref{la traza de HOM}  there exists a set $\{f_1,\dots,f_r\}$ with $f_i\in \HOM_{A}(P,A)$ and $\{p_1,\dots,p_r\}$, with $p_i\in P$ such that $\sum_{i=1}^rf_i(p_i)=a$.
In view of the equality $P=PB$,  we may write $p_i=\sum_{j}p_{ij}b_{ij}$, where $b_{ij}\in B$, so that $f_i(p_i)= \sum_j f_i(p_{ij}b_{ij})= \sum_{j}(b_{ij}f_i)(p_{ij})$. Then, since $b_{ij}f_i\in B\cdot\HOM_B(P,A)$ and $\psi$ is an isomorphism, there exists $q_{ij}\in Q$ such that $\psi(q_{ij})=b_{ij}f_i$, and consequently  $a=\sum_if_i(p_i)=\sum_{ij}(b_{ij}f_i)(p_{ij})=\sum_{ij}\psi(q_{ij})(p_{ij})=\sum_{ij}\pr{p_{ij}}{q_{ij}}$, and we are done.\\

Now we check the properties for $\nu$. Its $B$-bilinearity follows from the next computation with   $x,p\in P$, $q\in Q$ and $b\in B:$ 
\begin{align*}
 x (b\cdot [q,p])=&(xb)[q,p]\overset{\eqref{regla de multiplicar x[q,p]}}{=} \psi(q)(xb)p=(b\psi(q))(x)p=\psi(bq)(x)p\overset{\eqref{regla de multiplicar x[q,p]}}{=}x[bq,p].\\
 x([q,p]\cdot b)=&(x[q,p])b\overset{\eqref{regla de multiplicar x[q,p]}}{=} (\psi(q)(x)p)b=\psi(q)(x)(pb)\overset{\eqref{regla de multiplicar x[q,p]}}{=}x[q,pb].
\end{align*}

 To see that $\nu$ is graded of degree $e,$ consider $\{\gamma,\delta,\sigma,\tau\}\subset\Gamma$, $b\in B_\gamma$, $b'\in B_\delta$, $p\in P_\sigma$ and $q\in Q_\tau$. We know that $b[q,p]b'=[bq,pb']$. Then $\psi(bq)\in \HOM(P,A)_{\gamma\tau}$ and $pb'\in P_{\sigma\delta}$. Now for $x\in P_\rho$ we have that $x[bq,pb']\overset{\eqref{regla de multiplicar x[q,p]}}{=}\psi(bq)(x)pb'\in P_{\rho\gamma\tau\sigma\delta}$ and we are done.

 Next we show that  $Im\,\nu = B.$
Recall that we are identifying $B$ and $\beta(B)$. Observe that $Q\otimes_A P=BQ\otimes_A PB=B(Q\otimes_A P)B$, from which we have 
$$ Im\,\nu  = \nu(Q\otimes_A P)= B\nu(Q\otimes_A P)B\subseteq B \; \END(_BB) \; B = B.$$ In particular, $Im\,\nu$ is a  two-sided ideal in $B$ and 
\begin{equation}\label{eq:Im nu con B}
  Im\,\nu  = B \,  Im\,\nu =  Im\,\nu \, B = B\, Im\,\nu \, B.\end{equation}  Moreover,
 $$ Im\,\nu \; \END(_BB) = B Im\,\nu B \; \END(_BB) \subseteq  B Im\,\nu B = Im\,\nu ,$$ showing that $Im\,\nu$ is a right ideal in $\END(_BB).$ Consequently, we have the following inclusion of two-sided ideals  $\END(_BB) Im\,\nu \subseteq \END(_BB) B. $  

Observe that $P=P\cdot Im\,\nu .$ Indeed, 
take an arbitrary $p\in P$. Since $P\in A\gr$ we may write $p=\sum_ia_ip_i,$ and $a_i=\sum_j\pr{p'_{ij}}{q_j},$ because $\mu$ is onto. Then $$p=\sum_{i,j}\pr{p'_{ij}}{q_j}p_i\overset{\eqref{Asociatividad facil}}{=}\sum_{i,j}p'_{ij} [q_j,p_i]\in P\cdot Im\,\nu.$$
So that, $P\subseteq P\cdot Im\,\nu \subseteq P,$ proving that $P = P\cdot Im\,\nu .$

Note that $P = P Im\,\nu = P \,  (\END(_BB) \,  Im\,\nu )$ and by Lemma~\ref{PJ es P} we have that $B\subseteq \END(_BB)Im\,\nu ,$ which implies 
$\END(_BB)B\subseteq \END(_BB)Im\,\nu$. Since we saw already that  $\END(_BB)Im\,\nu \subseteq \END(_BB)B,$ we obtain the equality
$\END(_BB)Im\,\nu = \END(_BB)B.$ Hence, 
$$B=B\; \END(_BB)\; B= B\; \END(_BB)\; Im\,\nu=B \;  Im\,\nu \overset{\eqref{eq:Im nu con B}}{=} Im\,\nu ,$$ as desired.

It remains  to prove that
\begin{align}\label{la otra []q}
 [q,p]q'=q\pr{p}{q'}
\end{align} for all $p\in P$, $q,q'\in Q.$
   Let  $x\in P$. Recalling that $[q,p]\in B$, for any $p\in P$ and $q\in Q$ and having in mind the actions  of $A$ and $B$ on homomorphisms, we have
\begin{align*}
	\psi([q,p]q')(x)=&([q,p]\psi(q'))(x)=\psi(q')(x[q,p])\overset{\eqref{Asociatividad facil}}{=}\psi(q')\left(\pr{x}{q}p\right)=\\
	=&\pr{x}{q}\psi(q')(p)\overset{\eqref{tracemaps}}{=}\pr{x}{q}\pr{p}{q'},\\
	\psi(q\pr{p}{q'})(x)=&\left(\psi(q)\pr{p}{q'}\right)(x)=(\psi(q)(x))\pr{p}{q'}\overset{\eqref{tracemaps}}{=}\pr{x}{q}\pr{p}{q'},
\end{align*} which proves \eqref{la otra []q}.

Thus, $(A,B,P,Q,\mu,\nu)$ is a graded Morita context. As a consequence of the previous results, we have the following theorem.

\begin{theorem}\label{Car Equiv}
Let $A, B$ be idempotent  torsion-free graded rings, and let $F:A\gr\flechag B\gr$ and  $G:B\gr\flechag A\gr$ be inverse graded category equivalences. Set $P=G(B)$ and $Q=F(A)$. Then the following assertions hold.
 \begin{enumerate}
  \item $_AP_B$ and $_BQ_A$ are graded bimodules.
  \item $_AP\in A\gr$, $Q_A\in \grd A$, $_BQ\in B\gr$, $P_B\in \grd B,$ and  $_AP$ and  $_B Q$  are generators.
  \item There are natural graded isomorphisms
 \begin{eqnarray*}
 G&\cong_{gr-nat}& A\cdot \HOM_B(Q,-),\\
F&\cong_{gr-nat}& B\cdot\HOM_A(P,-).
\end{eqnarray*}
\item There exists a graded Morita context $(A,B,P,Q,\mu,\nu)$, with surjective trace maps.
 \end{enumerate}
\end{theorem}

We recall from \cite[p. 45]{GS} that a Morita context $(R,S,M,N)$ is said to be {\it non-degenerate} if the modules $_RM_S$, $_SN_R$ are torsion-free on both sides and the four pairings given by $\pr{-}{-}$ and $[-,-]$ are faithful (in the sense that $[n,M]=0$ implies $n=0$, for example).

\begin{corollary}\label{El contexto es nodegenerado}
	The graded Morita context $(A,B,P,Q,\mu,\nu)$ in Theorem~\ref{Car Equiv} is non-degenerate.
\end{corollary}
\begin{proof} Since we already know that $_AP_B$ and $_BQ_A$  are unital and torsion-free on both sides, we only need to check the faithfullness of the four pairings. We recall that $[Q,P]=B$ and $\pr{P}{Q}=A$.
	For all $p\in P$ and $q\in Q$, we have
	\begin{eqnarray*}
		\pr{P}{q}=0&\Rightarrow& Q\pr{P}{q}=0\Rightarrow [Q,P]q=0\Rightarrow Bq=0\Rightarrow q=0,\\
		\pr{p}{Q}=0&\Rightarrow& \pr{p}{Q}P=0\Rightarrow p[Q,P]=0\Rightarrow pB=0\Rightarrow p=0,\\
		\left[Q,p\right]=0 &\Rightarrow& P[Q,p] \Rightarrow \pr{P}{Q}p=0 \Rightarrow Ap=0 \Rightarrow p=0,\\
		\left[q,P\right]=0 &\Rightarrow& [q,P]Q\Rightarrow q\pr{P}{Q}=0 \Rightarrow qA=0 \Rightarrow q=0,
	\end{eqnarray*} as desired.
\end{proof} 

 As we mentioned  in Remark~\ref{torsion en tor} the graded bimodules $P\otimes_B Q$ and $Q\otimes_A P$ are unital, but not necessarily torsion-free. Nevertheless, the kernels of $\mu$ and $\nu$ have a very good behavior, as one may see in the next corollary.

\begin{corollary}\label{los nucleos de mu y nu son de torsion}
The trace maps $\nu$ and $\mu$ have torsion kernels.
\end{corollary}
\begin{proof}
We shall prove it for $\mu$, and  for $\nu$ the argument is completely analogous.

Suppose that $\mu(\sum_i p_i\otimes q_i)=\sum_i \pr{p_i}{q_i}=0$, with $p_i\in P$ and $q_i\in Q$. For any $a\in A$, write  $a=\sum_j\pr{x_j}{y_j}$ with $x_j\in P$ and $y_j\in Q$. Then
\begin{eqnarray*}
a\left(\sum_i p_i\otimes q_i\right)&=&\sum_i a p_i\otimes q_i=\sum_i \left(\sum_j\pr{x_j}{y_j}\right) p_i\otimes q_i=\\ &=&
\sum_i \sum_j(\pr{x_j}{y_j}p_i)\otimes q_i \overset{\eqref{Asociatividad facil}}{=} \sum_i\left( \sum_j(x_j[y_j,p_i])\otimes q_i\right)=\\ &=&
\sum_i\left( \sum_jx_j\otimes [y_j,p_i]q_i\right)\overset{\eqref{la otra []q}}{=} \sum_jx_j\otimes \sum_i y_j\pr{p_i}{q_i}=\\ &=&
\sum_jx_j\otimes y_j\sum_i\pr{p_i}{q_i}=0.
\end{eqnarray*}
Hence $a \sum_i p_i\otimes q_i=0$ for all $a\in A$.
\end{proof}

\section{From graded Morita contexts to graded category equivalences}\label{sec:FromMoritaContextToEquiv}

 The aim of this section is to prove a converse  of the result obtained in the  previous section; that is, to show that any  graded Morita context $(A,B,P,Q,\mu,\nu)$ with surjective trace maps, where $A, B$ are idempotent torsion-free graded rings and $_A P_B$, $_BQ_A$ are torsion-free  unital bimodules, induce an equivalence between $A\gr$ and $B\gr$. To this end, we will present some previous results within the more general framework of unital rings and modules, which are not necessarily torsion-free. These results will be useful for the generalization that we will consider in the next section.   We note that in this more general setting the categories of graded unital torsion-free modules $A\gr$, $B\gr$, $\grd A$ and $\grd B$ not necessarily contain any of $A,\,B,\,P,\,Q$.

\begin{proposition}\label{primeras propiedades del contexto}
Let $A, B$ be idempotent  graded rings and $_A P_B$, $_BQ_A$  unital bimodules, such that there is a graded Morita context $(A,B,P,Q,\mu,\nu)$ with surjective trace maps. Then
\begin{enumerate}
\item $_A P$, $_BQ$, $Q_A$ and $P_B$, generate $_AA$, $_BB$, $A_A$, $B_B$, respectively. Hence $_AP/t_A(P)$, $_BQ/t_B(Q)$, $P_B/t_B(P)$ and $Q_A/t_A(Q)$ are generators of their respective categories $A\gr$, $B\gr$, $\grd A$ and $\grd B.$
\item $\ker \mu$ and $\ker\nu$ are torsion graded bimodules.
\item The maps $$_BQ_A \flecha  B \cdot \HOM_A(P,A),\;\;\; q\mapsto \pr{-}{q}\;\;\; \text{and}\;\;\; _B Q_A \flecha  \HOM_B(P,B)\cdot A,\;\;\; q\mapsto [q, -],$$ $$_A P_B \flecha  A \cdot \HOM_B(Q,B),\;\;\; p\mapsto [-, p]\;\;\; \text{and}\;\;\; _A P_B \flecha  \HOM_A(Q,A) \cdot B,\;\;\; p\mapsto \pr{p}{-}$$ are graded epimorphisms of degeree $e$ of graded bimodules, whose kernels are, respectively, left $B$-, right $A$-, left $A$- and right $B$-torsion. 
\end{enumerate}
 If in addition $A,$  $B,$ $_A P_B$ and $_BQ_A$ are all  (left and right) torsion-free, then
 
\begin{enumerate}
	\setcounter{enumi}{3} 
  \item
\begin{itemize}
	\item[] $ Q \cong  B\cdot \HOM_A(P,A)$ and $Q \cong  \HOM_B(P,B)\cdot A$
	\item[]  $ P\cong  A\cdot\HOM_B(Q,B)$ and $P\cong  \HOM_A(Q,A)\cdot B$.
\end{itemize}
\item The context is non-degenerate.
\end{enumerate}
\end{proposition}
\begin{proof}
(1) For each $q\in Q$ we have that  $\pr{-}{q}\in \HOM_A(P,A).$ Then taking  any $a\in A$, we may write $a=\sum_i\pr{p_i}{q_i}$ where $p_i\in P$ and $q_i\in Q$. So $a=\sum_i\pr{-}{q_i}(p_i)\in Tr_A(P)$.
Hence $A= Tr_A(P),$ which by  Proposition~\ref{unital genera unital} shows that $_A P$ generates $A$. 

 Now, it is clear that $_AP/t_A(P)$ belongs to $A\gr.$ To see that $_A P/t_A(P)$ is a generator, consider any $U\in A\gr$. We first check that $Tr_U(P)=U$. Take any $u\in U$ and write $u=\sum_ia_iu_i,$ $a_i \in A,$ $u_i\in U.$ On the one hand, we have $a_i=\sum_jf_{ij}(p_{ij})$, for $f_{ij}\in \HOM_A(P,A)$ and $p_{ij}\in P$, because $Tr_A(P)=A$. On the other hand, considering the right multiplication $\cdot u_i:A\flecha U$,  we have that $\sum_{ij}f_{ij}(-)\cdot u_i\in \HOM_A(P,U)$ and  $u=\sum_ia_iu_i=\sum_{ij}f_{ij}(p_{ij})\cdot u_i\in Tr_U(P)$. Now,  for any $f\in \HOM_A(P,U)$, we have that $t_A(P)\subseteq \ker\,f$, thanks to the fact that  $U$ is torsion-free, so that there exists  $\hat f\in \HOM_A(P/t_A(P),U)$, such that $Im\,f= Im\;\hat f$. Then $U=Tr_U(P)=Tr_{U}(P/t_A(P))$.  By Corollary~\ref{generador y traza} we conclude that $P/t_A(P)$ is a generator. The proofs of the other assertions in (1)  are similar.

(2) The proof is exactly the same as in Corollay~\ref{los nucleos de mu y nu son de torsion}.

(3) We prove the first  graded epimorphism, the proof of the others being  analogous. Consider $\psi:Q\flecha  \HOM_A(P,A)$ given by $\psi(q)(p)=\pr{p}{q},$  $p\in P$ and $q\in Q.$

Now we shall prove the following properties of $\psi$.

(i) $\psi$ is a $\bimod{B}{A}$-bimodule homomorphism. Let $a\in A$, $b\in B$ and $q\in Q$. Then for any $p\in P$,
\[ \psi(bqa)(p)=\pr{p}{bqa}=\pr{pb}{q}a=\psi(q)(pb)a= (b\psi(q))(p)a = (b\psi(q)a)(p) , \]
by the properties of the trace maps and  the definition of the actions of $A$ and $B$ on the  homomorphisms.

(ii) $\psi$ is a graded homomorphism of degree $e\in \Gamma$; so that, it is a bimodule homomorphism. Take $q\in Q_\sigma$. Then, for any $\tau\in \Gamma$ we have $\psi(q)(P_\tau)=\pr{P_\tau}{q}\subset A_{\tau\sigma}$ because $\mu$ is graded. So that $\psi(q)\in  \HOM_A(P,A)_\sigma$, proving that the degree of $\psi$ is $e$.

(iii) $\ker\, \psi$ is left $B$-torsion. Suppose that $q\in \ker \, \psi .$ Then $0=\psi(q)(P)=\pr{P}{q}$ and  $0=Q\pr{P}{q}=[Q,P]q=Bq .$ 

(iv) $Im\,\psi = B\cdot \HOM_A(P,A)$. Take $q\in Q$ and write $q=\sum_ib_iq_i$. Then obviously we have $\psi(q)=\sum_ib_i\psi(q_i)$, so that  $Im\,\psi \subseteq B\cdot \HOM_A(P,A)$. To see the opposite inclusion, take $b\in B$ and $g\in \HOM_A(P,A)$ and write $b=\sum_i[q_i,p_i]$ with $q_i\in Q$ and $p_i\in P$.  Then for any $x\in P$ we have 
\begin{align*}(bg)(x)&=g(xb)=g(x\sum_i[q_i,p_i])=g(\sum_i\pr{x}{q_i}p_i)=\sum_ig(\pr{x}{q_i}p_i)\\&=\sum_i \pr{x}{q_i}g(p_i)=  \sum_i \psi(q_i)(x)\cdot g(p_i)=\psi(\sum_i q_ig(p_i))(x). 
\end{align*} Consequently, $bg= \psi(\sum_i q_ig(p_i))\in Im\,\psi, $ as desired. 
This proves the first epimorphism in  (3), the proof of others being similar.

(4) This is an immediate consequence of (3). 

(5) The proof  is exactly the same as in Corollary~\ref{El contexto es nodegenerado}, as it depends only on the associativity properties of the trace maps and the fact that all graded bimodules are torsion-free.
\end{proof}

\begin{notation}\label{notation}
For a given graded Morita context  $(A,B,P,Q,\mu,\nu)$ with  idempotent $A,$ $B$, surjective trace maps and unital  bimodules, we write $\psi(q)=\pr{-}{q}$, for any $q\in Q$; that is $\psi:Q\flecha  B\cdot \HOM_A(P,A)$ is a $\bimod{B}{A}$-bimodule graded epimorphism from the proof of  Proposition~\ref{primeras propiedades del contexto}. In fact, we shall use both symbols,  $\psi(q)$ and  $\pr{-}{q}$.
\end{notation}

 Now we prove a particular case of some hom-tensor relations.

\begin{proposition}\label{hom y tor}
Let $A, B$ be idempotent graded rings and $_A P_B$, $_BQ_A$  unital bimodules, such that there is a  graded Morita context $(A,B,P,Q,\mu,\nu)$ with surjective trace maps. Then for any torsion-free unital modules  $_AK$ and $_BL$, we have the following.
\begin{enumerate}
\item The canonical map
\[\varphi:P\otimes_B B\cdot\HOM_A(P,K)\flecha A\cdot\HOM_A(A,K), \]
given by $\varphi(p\otimes f)(x)=  f(xp)=  xf(p),$ with $p \in P,$ $f\in  B\cdot\HOM_A(P,K), $ $x \in A,$ is a graded epimorphism of degree $e$ of graded left $A$-modules with torsion kernel.
\item The canonical map
\[ \gamma:B\otimes_B L\flecha B\cdot\HOM_A(P,P \otimes_BL), \]
given by $\gamma(b\otimes l)(p)=pb\otimes l,$ with $b \in B,$ $l \in L,$ $p \in P,$ is a graded epimorphism of degree $e$ of graded left $B$-modules with torsion kernel.
\end{enumerate}
\end{proposition}

\begin{proof} (1) First, note that, for any $p\otimes f \in P\otimes_BB\cdot\HOM_A(P,K)$, we have that $f(p)=\sum_ia_ik_i$ thanks to the fact that the $A$-module $K$ is unital. Moreover, since $_A K$ is torsion-free, using Proposition~\ref{isomorfismo clave} we may identify   $f (p)$ with  $\cdot \sum_ia_ik_i \in A\HOM_A(A,K)$ so that our $\varphi $ maps 
  $p\otimes f$ to $\cdot f(p) .$ Hence $Im\,\varphi\subseteq A\HOM_A(A,K)$ and $\varphi $ is well-defined.
  
  Now,  take arbitrary $a \in A$ and  $p\otimes f \in P\otimes_B B\cdot\HOM_A(P,K).$ Then  for any $x\in A$ we have that 
  $$(a \cdot \varphi (p \otimes f) ) (x)=
   \varphi (p \otimes f)  (xa) =  f(xap)=  xf(ap)=   \varphi (ap \otimes f)  (x),$$ 
   showing that  $\varphi $ is a left $A$-module map. 
   
  Next, for $p \in P_{\tau}, $ $f \in (B\cdot\HOM_A(P,K))_{\sigma}$ and $x \in A_{\rho},$ we  have that  $f(p) \in K_{\tau \sigma }$ and   
   $\varphi (p \otimes f) (x) = x f(p) \in K_{\rho \tau \sigma} ,$ which proves that    $\varphi $ is graded of degree $e,$ as $p \otimes f \in (P \otimes B\cdot\HOM_A(P,K) )_{\tau \sigma}.$

   In order to see that $\varphi $ is surjective let $f \in A\cdot\HOM_A(A,K) .$ By  Proposition~\ref{isomorfismo clave} $f = \cdot k $ for some $k \in K.$ Since  $_A P$ is a generator, By Remark~\ref{la traza de HOM} we may write $k = \sum f_i (p_i) $ for some $f_i \in \HOM_A(P,K) $ and $p_i \in P.$ Then $\sum \varphi (p_i \otimes f_i)(x) = \sum f_i (x p_i) =
   x \sum f_i ( p_i) = x k = f(x). $
   
   To complete the proof of (1), it remains to show that  $\ker \, \varphi $ is left $A$-torsion. Let $\sum p \otimes f \in \ker \, \varphi .$ Then
    \begin{equation}\label{eq:xfp zero}
    x \sum  f(p) =0
    \end{equation} for all $x \in A.$ Taking any $a \in A,$ write $a = \sum _i \pr{p_i}{q_i}$ with $p_i \in P$ and $q_i \in Q.$ Then 
   $$a \sum (p \otimes f) = \sum \sum_i \pr{p_i}{q_i}p \otimes f =  \sum _i  p_i \otimes \sum [q_i, p]  f $$ and since 
   $ \sum [q_i, p]  f  (y) =  \sum   f  (y [q_i, p]) = \pr{y}{q_i} \sum f(p)= 0$ for each $i$ and all $y \in P,$ 
   in view of \eqref{eq:xfp zero}. This means that $\sum [q_i, p]  f = 0$ and, since $A\in A$ was chosen arbitrarily, it follows that $\sum p \otimes f =0,$ completing  our proof of (1).

   (2) Taking arbitrary $b \in B,$ $l\in L$ and $p\in P.$ Since $B$ is idempotent,  $b = \sum b_1 b_2,$ which we write in a simplied form $b = b_1 b_2.$ Then 
   $$\gamma (b \otimes l) (p) = p b_1 b_2 \otimes l = \gamma (b_2 \otimes l) (p b_1) =   b_1 \gamma (b_2 \otimes l) (p ),$$  showing that $Im \, \gamma \subseteq B\cdot\HOM_A(P,P \otimes_BL).$

   Next $$\gamma (b' (b \otimes l))(p) = p b' b \otimes l= \gamma( b \otimes l) (p b')  = b' \gamma( b \otimes l) (p ),    $$
   for all $b,b' \in B,$ $l\in L$ and $p \in P,$ showing that $\gamma $ is a homomorphism of left $B$-modules.

   For arbitrary $b \in B_{\sigma}, $ $l \in L_{\rho}$ and $p \in P_{\tau}$ we have 
   $\gamma (b \otimes l) (p) = pb \otimes l \in (P \otimes L)_{\tau \sigma \rho,} $ so that $\gamma $ is graded of degree $e.$

   Now, take any $f \in \HOM_A(P,P \otimes_BL) $
   and $b \in B .$ Write 
   $b = \sum _i [q_i, p_i]$ with $q_i \in Q$ and $p_i \in P.$ Then 
   $$(bf)(x) = f(x \sum _i [q_i, p_i])=
  \sum _i f(\pr{x}{q_i}p_i)  
  =
  \sum _i \pr{x}{q_i} f(p_i).  $$ Writing 
  $f(p_i) = \sum _j p'_{ij} \otimes l_j,$ $p'_{ij} \in P, $ $l_j \in L,$ we compute, for any $x\in P,$ 
  \begin{align*}
    \gamma (\sum _{i,j}[q_i, p'_{ij}] \otimes l_j) (x) & = 
  \sum _{i,j}x[q_i, p'_{ij}] \otimes l_j =
  \sum _{i,j}\pr{x}{q_i} p'_{ij} \otimes l_j\\ &= \sum _{i}\pr{x}{q_i} \sum _j p'_{ij} \otimes l_j = \sum _{i}\pr{x}{q_i} f (p_i)  =(bf) (x),
  \end{align*}  showing that $\gamma $ is surjective.

	It remains to prove that $\ker \, \gamma $ is left $B$-torsion. Let $\sum b \otimes l \in \ker \, \gamma.$  Then
  \begin{equation}\label{eq: b tensor l in kernel}
    0=\gamma (\sum b \otimes l) (x) = \sum xb \otimes l
  \end{equation} for all $x \in P.$ Since 
  $[-,-]$ is surjective, write $b = \sum _i [q_i, p_i],$ with $q_i \in Q$ and $p_i \in P.$ Then for any $q' \in Q$ and $p' \in P$ we see that 

\begin{align*}
[q', p']\sum \sum _i (q_i \otimes _A p_i) \otimes l&= \sum \sum _i ([q', p']q_i \otimes _A p_i) \otimes l=
  \sum \sum _i (q' \pr{p'}{q_i}  \otimes _A p_i) \otimes l\\
  &= 
  \sum \sum _i (q'   \otimes _A \pr{p'}{q_i} p_i) \otimes  l = 
    \sum (q'   \otimes _A p'\sum _i [q_i, p_i]) \otimes  l\\
    &= 
   \sum (q'   \otimes _A p'b) \otimes  l = 
   q'   \otimes _A \sum (p'b \otimes  l) =0
\end{align*} thanks to \eqref{eq: b tensor l in kernel}. Since $q' $ and $p'$ were chosen arbitrarily and $[-,-]$ is surjective, we obatin that
$B \sum _i (q_i \otimes _A p_i) \otimes l =0 ,$ so that  $ \sum _i (q_i \otimes _A p_i) \otimes l $ is a left $B$-torsion element of $(Q \otimes _A P) \otimes _B L.$ Applying the left $B$-module map 
$$[-,-]\otimes 1_L: (Q \otimes _A P) \otimes _B L \to B \otimes _B L,$$ to $ \sum _i (q_i \otimes _A p_i) \otimes l , $ we conclude that our element  
$b \otimes l =  \sum _i [q_i , p_i] \otimes l $  is left $B$-torsion, as desired. \end{proof}

We shall need the next result, which  can be extracted from the proof of \cite[Proposition 2.6(iii)]{GS}. We give a complete argument  for the reader's convenience.

\begin{lemma}\label{gamma con el producto}
Let $A, B$ be idempotent  graded rings and $_A P_B$, $_BQ_A$  unital graded bimodules, such that there is a graded Morita context $(A,B,P,Q,\mu,\nu)$ with surjective trace maps. For any unital graded module $_A U$ we have that the map
\[ \eta: B\cdot\HOM _A(P,A)\otimes_AU\flecha B\cdot\HOM_A(P,U), \]
given by $\eta(bf\otimes u)(p)=((bf)(p))u,$ is a graded epimorphism of degree $e,$  with torsion kernel, of graded left $B$-modules.
\end{lemma}
\begin{proof}
To check $B$-linearity, take $b',b\in B$, $u\in U$, $f\in  B\cdot\HOM(P,A)$ and $p\in P$. Then $\eta(b'bf\otimes u)(p)=(b'bf)(p)u=(bf)(pb')u =\eta(bf\otimes u)(pb')=b'\eta(bf\otimes u)(p)$.   

Now, for  $b\in B_{\sigma}$, $u\in U_{\rho}$,  $f\in  B\cdot\HOM(P,A)_{\tau}$ and $p\in P_{\delta}$ we have that $\eta(bf\otimes u)(p)=(bf)(p)u\in U_{\delta\sigma\tau\rho}$, so that $\eta$ is a graded homomorphism of degree $e$.    

 We shall see that $\eta$ is onto,  proceeding as in the proof of Proposition~\ref{primeras propiedades del contexto}. Take arbitrary $b\in B$ and $g\in \HOM_A(P,U)$, and  write $b=\sum_i[q_i,p_i]$ with $q_i\in Q$ and $p_i\in P$.  Note that, for any $x\in P$ we have $$(bg)(x)=g(xb)=g(x\sum_i[q_i,p_i])=g(\sum_i\pr{x}{q_i}p_i)=\sum_ig(\pr{x}{q_i}p_i)=\sum_i \pr{x}{q_i}g(p_i).$$ In order to produce a preimage of $bg$ consider $\sum_i \pr{-}{q_i}\otimes g(p_i)=\sum_i \psi(q_i)\otimes g(p_i)\in  B\cdot\HOM(P,A)\otimes_AU$. Then, for any $x\in P$,
\[\eta\left(\sum_i \pr{-}{q_i}\otimes g(p_i)\right)(x) = \sum_i \pr{x}{q_i} g(p_i)=(bg)(x),   \] provig that $\eta $ is surjective.

Now we show that $\ker\eta$ is a torsion left $B$-module.
Suppose that there are $f\in \HOM(P,A)$, $u\in U$ and $b\in B$ sucht that $\eta(\sum bf\otimes u)=0$. 
 By (3) of Proposition~\ref{primeras propiedades del contexto}, using Notation~\ref{notation}, the map $\psi:Q\flecha A\cdot \HOM(P,A)$ is surjective, so that  there exists  $q\in Q$ such that $\psi(q)=bf. $ By the definition of $\eta$ we have $0= \psi(q) (p)u=\pr{p}{q}u$ for all $p\in P$.

Now for any $y\in Q$ and $x\in P$, having in mind that $\psi$ is a $\bimod{B}{A}$-bimodule map, we see that
\begin{eqnarray*}
[y,x]\sum (\psi(q)\otimes _A u)&=&\sum [y,x]\psi(q)\otimes u=\sum \psi([y,x]q)\otimes u=\sum \psi(y\pr{x}{q})\otimes u\\
&=& \sum \psi(y)\pr{x}{q}\otimes u=\psi(y)\otimes \sum \pr{x}{q} u=0.
\end{eqnarray*}
Since the trace map $\nu $ is surjective, it follows that   $B(\psi(q)\otimes u)=0,$ proving that the kernel is torsion.
\end{proof}

\begin{corollary}\label{cor:Hom elimina la torsion}
 Let $A, B$ be idempotent  graded rings and $_A P_B$, $_BQ_A$  unital graded bimodules, such that there is a graded Morita context $(A,B,P,Q,\mu,\nu)$ with surjective trace maps. For any unital graded module $_A U$ we have that
 \[ B\cdot\HOM _A(P,U)\cong B\cdot\HOM _A(P,U/t_A(U)) \] as graded left $B$-modules.
\end{corollary}
\begin{proof}
  Set $H=B\cdot\HOM(P,A)$ and  consider the exact sequence $0\flecha t_A(U)\flecha U\flecha U/t_A(U)\flecha 0$. Note that $H$ is a graded  right $A$-module under the action given by $(fa)(x)=f(x)a$,  $f\in H$, $a\in A,$  $x\in P.$ In fact,  $H$ is  unital as a graded right $A$-module, by Proposition~\ref{primeras propiedades del contexto}(3). 
 So that,  $H\otimes_A t_A(U)=0$. Now, by appling the right exact functor $H\otimes_A-$ to the above sequence, we get $H\otimes U\cong H\otimes U/t_A(U)$. By Lemma~\ref{gamma con el producto}, we are done. 
  \end{proof}

\begin{proposition}\label{endomorfismos y producto tensorial como funtor}
Let $A, B$ be idempotent torsion-free graded rings and $_A P_B$, $_BQ_A$ torsion-free  unital bimodules, such that there is a graded Morita context $(A,B,P,Q,\mu,\nu)$ with surjective trace maps. Then
\begin{enumerate}
\item The canonical graded maps $A\flecha \END(_BQ)$, $A\flecha \END(P_B)$, $B\flecha \END(Q_A)$ and $B\flecha \END(_AP)$ induce graded ring isomorphisms $A\cong A\cdot \END(_BQ)$, $A\cong \END(P_B)\cdot A$, $B\cong \END(Q_A)\cdot B$ and $B\cong B\cdot\END(_AP)$.
\item The functors $P\otimes_B -/t_A(P\otimes_B -):B\gr \flecha A\gr$ and $B\cdot \HOM_A(P,-):A\gr\flecha B\gr$ are graded inverse equivalences of categories.
\item The functors $P\otimes_B -/t_A(P\otimes_B -):B\gr \flecha A\gr$ and $Q\otimes_A -/t_B(Q\otimes_A -):A\gr \flecha B\gr$ are graded inverse equivalences of categories.
\end{enumerate}
There are also graded category equivalences between $\grd A$ and $\grd B$ similar to those in (2) and (3).
\end{proposition}
\begin{proof}
(1) We prove the first isomorphism. The others are similar. As $_BQ_A$ is torsion-free at both sides, the right multiplication map by elements of $A$ is injective. To see surjectivity, take $a\in A$ and $f\in \END(_B Q)$. Now, write $a=\sum_i\pr{p_i}{q_i}$ with $p_i\in P$ and $q_i\in Q$. Then for any $x\in Q$ we have
\begin{eqnarray*}
(af)(x)&=&f(xa)=f(x\sum_i\pr{p_i}{q_i})=f(\sum_i [x,p_i]q_i)=\sum_i [x,p_i]f(q_i)= \\ 
&=&\sum_i x\pr{p_i}{f(q_i)}=x\sum_i\pr{p_i}{f(q_i)}
\end{eqnarray*}
so $af=\cdot \sum_i\pr{p_i}{f(q_i)}\in A$.

	(2) Let $K\in A\gr$ and  
	consider the graded surjective map $\varphi$ from Proposition~\ref{hom y tor}.  Then we have
	\begin{eqnarray*}
		P\otimes _B B\cdot \HOM_A(P,K) \flechadecor{\varphi} 
		A\cdot\HOM(A,K)\cong K,
	\end{eqnarray*}
	where the last graded isomorphism is $\chi_K,$ given by  Proposition~\ref{isomorfismo clave}. Since  $\varphi$ has torsion kernel and $K$ is torsion-free, we obtain that
	$\ker\varphi=t_A(P\otimes _B B\cdot \HOM_A(P,K))$ and, consequently, 
	$P\otimes _B  B\cdot \HOM_A(P,K)/ t_A(P\otimes _B B\cdot \HOM_A(P,K)) \cong  K. $
	
	On the other hand, for $L\in B\gr$ we have by (2) of Proposition~\ref{hom y tor} the graded epimorphism
	\[ B\otimes_B L \flechadecor{\gamma}B\cdot \HOM_A(P,P\otimes_B L),\]
	whose kernel is torsion.
	By Lemma~\ref{el HOM siempre es torsionfree} we know that $B\cdot \HOM_A(P,P\otimes_B L)$ is torsion-free, so that the kernel of $\gamma $ is exactly $t_B(B \otimes_B L).$ Similarly, it is easy to see that the kernel of the canonical map $B \otimes _B L \to L$ is torsion-free and, since $L$ is torsion-free, the kernel  is $t_B(B \otimes_B L).$ Therefore,
	$$L\cong B\otimes_B L/t_B(B \otimes_B L)\cong  B\cdot \HOM_A(P,P\otimes_B L)\cong  B\cdot \HOM_A(P,\frac{P\otimes_B L}{t_A(P\otimes_B L)}),$$ where the last isomorphism comes from Corollary\ref{cor:Hom elimina la torsion}. This gives us (2).
	
(3) By (2) it is enough to prove that $Q\otimes_A -/t_A(Q\otimes_A -)$ is graded isomorphic to $B\cdot\HOM_A(P,-)$. By (4) of Proposition~\ref{primeras propiedades del contexto}, for any $U\in A\gr$ we have that $Q\otimes_A U\cong B\cdot\HOM_A(P,A)\otimes_A U$. By Lemma~\ref{gamma con el producto} the composition of this isomorphism with $\eta$, that is $Q\otimes_A U\flecha B\cdot\HOM_A(P,U),$ is surjective with torsion kernel. Since $ B\cdot\HOM_A(P,U)$ is torsion-free, we obtain that $Q\otimes_A U/t_B(Q\otimes_A U)\cong B\cdot\HOM_A(P,U),$ completing our  proof.
\end{proof}

\begin{remark}\label{rem:item(2)NoSiFaltaSinTorcion} 
  Observe that in the proof of item (2) of 
  Proposition~\ref{endomorfismos y producto tensorial como funtor} we did not use  neither the restrictions on  $A$ and $B$ of being torsion-free, nor the facts that  the bimodules $P$ and $Q$ are torsion-free. 
\end{remark}

\begin{remark}\label{remark:NaturalGradedIsos} As we have shown in the proof of Proposition~\ref{endomorfismos y producto tensorial como funtor}, there is a natural graded isomorphism
  $P\otimes_B -/t_A(P\otimes_B -)\cong  A \cdot \HOM_B(Q,-)  $. 
Analogously, there is a natural graded isomorphism
 $Q\otimes_A -/t_B(Q\otimes_A -)\cong B\cdot \HOM_A(P,-) $.
\end{remark}

Summarizing the results of this section we have the following theorem (see \cite[Proposition 2.6]{GS}).

\begin{theorem}\label{contextos a equivalencias final}
Let $A, B$ be idempotent torsion-free graded rings and $_A P_B$, $_BQ_A$ torsion-free  unital bimodules, such that there is a graded Morita context $(A,B,P,Q,\mu,\nu)$ with surjective trace maps. Then
\begin{enumerate}
\item $_A P$, $P_B$, $_BQ$ and $Q_A$ are generators of their corresponding categories.
\item The canonical graded maps $A\flecha \END(_BQ)$, $A\flecha \END(P_B)$, $B\flecha \END(Q_A)$ and $B\flecha \END(_AP)$ induce graded ring isomorphisms $A\cong  A\cdot\END(_BQ)$, $A\cong  \END(P_B)\cdot A$, $B\cong  \END(Q_A)\cdot B$ and $B\cong  B\cdot\END(_AP)$.
\item 
The functors
\begin{itemize}
 \item[] $P\otimes_B -/t_A(P\otimes_B -):B\gr \flecha A\gr$ and 
 \item[] $Q\otimes_A -/t_B(Q\otimes_A -):A\gr \flecha B\gr$ 
\end{itemize} 
are inverse graded equivalences of categories.
\item There are graded natural isomorphisms:
\begin{enumerate}
 \item $P\otimes_B -/t_A(P\otimes_B -)\cong  A \cdot \HOM_B(Q,-)  $.
 \item $Q\otimes_A -/t_B(Q\otimes_A -)\cong B\cdot \HOM_A(P,-) $.
\end{enumerate}
\item There are graded isomorphisms $$_B Q\cong B\cdot \HOM_A(P,A)  \cong Q\otimes_A A/t_B(Q\otimes_A A),$$ and similar isomorphisms hold  for $Q_A$, $_AP$ and $P_B$.
\item The graded tensor products $P\otimes_B Q$ and $Q\otimes_A P$ have a structure of a $\Gamma$-graded rings given by
\[(p_1\otimes q_1)(p_2\otimes q_2)=p_1\otimes[q_1,p_2]q_2\quad\text{and}\quad (q_1\otimes p_1)(q_2\otimes p_2)=q_1\pr{p_1}{q_2} \otimes p_1.\]
\end{enumerate}
\end{theorem}
\begin{proof}
 The assertions (1) to (5) follow from Propositions~\ref{primeras propiedades del contexto} and \ref{endomorfismos y producto tensorial como funtor} and Remark~\ref{remark:NaturalGradedIsos}. The proof of (6) is an easy straightforward calculation.
\end{proof}

The following theorem is a simple and very usefull characterization of graded category equivalences.

\begin{theorem}\label{teo:UsefulCharacterization}
 Let $A, B$ be idempotent torsion-free graded rings. Then $A\gr$ and $B\gr$ are graded equivalent categories if and only if there exists a torsion-free  unital graded bimodule $_AP_B$ such that:
 \begin{enumerate}
  \item $_AP$ and $P_B$ are graded generators of their respective categories.
  \item The canonical graded map $B\flecha \END(_AP)$ induces a graded ring isomorphism $B\cong  B\cdot\END(_AP)$.
 \end{enumerate}
Moreover, in this case, the functors $P\otimes_B -/t_A(P\otimes_B -):B\gr \flecha A\gr$ and $B\cdot \HOM_A(P,-):A\gr\flecha B\gr$ are graded inverse equivalences of categories.
\end{theorem}

\begin{proof}
 Suppose first that $F:A\gr\flechag B\gr$ and  $G:B\gr\flechag A\gr$ are inverse graded category equivalences. Then with  $P=G(B)$ items (1) and (2) follow from  Theorem~\ref{Car Equiv} and Theorem~\ref{contextos a equivalencias final}.
 
 Conversely, suppose that (1), (2) hold and set $Q=B\cdot \HOM_A(P,A)\subset \HOM_A(P,A)$. Recall that  the right action of $A$ on $Q$ is given by $(fa)(x)=f(x)a$,  $f\in Q$, $a\in A,$  $x\in P,$ whereas the left $B$-action on $Q$ 
 is defined by $bf (x) = f(xb), $ $b \in B,$  endowing $Q$ with a unital graded  $\bimod{B}{A}$-bimodule structure. Moreover, $_BQ_A$ is a torsion-free bimodule by Lemma~\ref{el HOM siempre es torsionfree}.
 Now, it is easy to see that we may define graded bimodule maps such that, for $p,x\in P$ and $q\in Q$,
 \begin{eqnarray*}
  \mu:P\otimes_B Q\flecha A &\text{given by}& \mu(p\otimes q)=q(p),\\
  \nu:Q\otimes_A P \flecha \END(_AP) &\text{given by}& \nu(q\otimes p)(x)=q(x)p.
 \end{eqnarray*}
Condition (1) immediately implies  $Im\,\mu =A .$ Next we show that $Im\,\nu=B\cdot \END(_AP)$. Since $\nu $ is $B$-bilinear, we have  $Im\,\nu \subseteq B\cdot \END(_AP)$. For the converse inclusion observe that $P \cdot Im\,\nu =P.$ Indeed, since $\mu $ is onto, for any $p \in P$ we may write 
$$p = \sum _i a_i p_i= \sum _i  \sum _j q_{ij}(x_{ij}) p_i = \sum_{ij} \nu (q_{ij} \otimes p_i)(x_{ij}) = 
x_{ij} \cdot \sum_{ij} \nu (q_{ij} \otimes p_i),$$ where $a_i \in A, p_i, x_{ij} \in P, $ $q_{ij} \in Q,$ so that  $ P \subseteq P \cdot Im\,\nu .  $ Because the converse inclusion is obvious, we obtain the desired equality $P \cdot Im\,\nu =P.$

Note that thanks to the fact that  $P_B$ is a  generator, it follows from  condition (2) that   
$Tr_{B\cdot \END(_AP)}(P) = B\cdot \END(_AP).$ Therefore, taking an arbitrary 
 $\bar{b} \in B\cdot \END(_AP),$ we may write 
 $\bar{b} = \sum _i h_i (x_i)$ for some $h_i \in  \HOM(P_B,B\cdot \END(_AP)),$ $x_i \in P.$ Using the  equality $P \cdot Im\,\nu =P,$ we also write $x_i = \sum _j p_{ij} \cdot v_{ij}, $ where $p_{ij}\in P$ and $v_{ij} \in Im\,\nu .$ Consequently, 
 $$\bar{b} = \sum _i h_i (\sum _j p_{ij} \cdot v_{ij})=\sum _{ij} h_i ( p_{ij} ) v_{ij} \in  Im\,\nu ,$$ which proves the converse inclusion $ B\cdot \END(_AP) \subseteq Im\,\nu ,$ establishing the desired equality   $Im\,\nu=B\cdot \END(_AP)$. 
 
 The properties of $\mu $ and $\nu $ which gives a  Morita context 
  $(A,B,P,Q,\mu,\nu)$ are easily verified.   Thus, we have  a graded Morita context with surjective trace maps, and we are done by   Theorem~\ref{contextos a equivalencias final}.
\end{proof}

\begin{corollary}
  Let $A, B$ be idempotent torsion-free graded rings. Then $A\gr$ and $B\gr$ are graded equivalent categories if and only if the categories $\grd A$ and $\grd B$ are graded equivalent.
\end{corollary}
\begin{proof} Suppose that $A\gr$ and $B\gr$ are graded equivalent categories. Then tanks to Theorem~\ref{Car Equiv}, there exists a graded Morita context  $(A,B,P,Q,\mu,\nu)$ with surjective trace maps. Then by Theorem~\ref{contextos a equivalencias final} $Q_A$ is a generator of $\grd A,$ $_B Q$ is a generator of $B\gr $ and the canonical graded map $B\flecha \END(Q_A)$  induces a  graded ring isomorphism  $B\cong  \END(Q_A)\cdot B .$ Then a symmetric version of Theorem~\ref{teo:UsefulCharacterization} gives us a graded equivalence between  
   $\grd A$ and $\grd B .$ 
\end{proof}

\section{Generalization to idempotent not necessarily torsion-free graded rings}\label{sec:Generalization}

  We consider a graded Morita context $(A,B,P,Q,\mu,\nu)$ in which the graded rings $A$ and $B$ are idempotent and the bimodules $_A P_B$ and $_B Q_A$ are unital, but we do not suppose that any of them is torsion-free.
  
% \red{ ESTO NO HACE FALTA Note that, in this case, the categories of graded unital torsion-free modules $A\gr$, $B\gr$, $\grd A$ and $\grd B$ not necessarily contain any of $A,\,B,\,P,\,Q$. }
  
 % \red{ ESTO YA SE DIJO For rings we have to distinguish between left or right torsion. So we denote the left torsion of $A$ by $l(A)$ and the right one by  $r(A).$ Clearly we have that $A/l(A)\in A\gr$ etc.}
    
  We give the following result without assuming any kind of torsion-free condition.
  
  \begin{theorem}\label{teo:Generalization}
  	Let $A, B$ be idempotent graded rings and $_A P_B$, $_BQ_A$  unital bimodules, such that there is a graded Morita context $(A,B,P,Q,\mu,\nu)$ with surjective trace maps. Then
  	\begin{enumerate}
  		\item The canonical graded maps $A\flecha \END(_BQ)$, $A\flecha \END(P_B)$, $B\flecha \END(Q_A)$ and $B\flecha \END(_AP)$ induce graded ring isomorphisms $A/l(A)\cong A\cdot \END(_BQ)$, $A/r(A)\cong \END(P_B)\cdot A$, $B/r(B)\cong \END(Q_A)\cdot B$ and $B/l(B)\cong B\cdot\END(_AP)$.
  		\item The functors $P\otimes_B -/t_A(P\otimes_B -):B\gr \flecha A\gr$ and $B\cdot \HOM_A(P,-):A\gr\flecha B\gr$ are graded inverse equivalences of categories.
  		\item The functors $P\otimes_B -/t_A(P\otimes_B -):B\gr \flecha A\gr$ and $Q\otimes_A -/t_B(Q\otimes_A -):A\gr \flecha B\gr$ are graded inverse equivalences of categories.
  	\end{enumerate}
  	There are also graded category equivalences between $\grd A$ and $\grd B$ similar to those in (2) and (3).
  \end{theorem}
  \begin{proof}
  	(1) We give details for  the first map, the other cases being similar. 
  	
  	We begin by proving that the image of the canonical map is $ A\cdot \END(_BQ)$.  Take arbitrary $a\in A$ and $f\in \END(_B Q)$ and write $a=\sum_i\pr{p_i}{q_i}$ with $p_i\in P$ and $q_i\in Q$. Then for any $x\in Q$ we have
  	\begin{eqnarray*}
  		(af)(x)&=&f(xa)=f(x\sum_i\pr{p_i}{q_i})=f(\sum_i [x,p_i]q_i)=\sum_i [x,p_i]f(q_i)= \\ 
  		&=&\sum_i x\pr{p_i}{f(q_i)}=x\sum_i\pr{p_i}{f(q_i)}
  	\end{eqnarray*}
  	so $af=\cdot \sum_i\pr{p_i}{f(q_i)}\in A,$ showing that the image is contained in $ A\cdot \END(_BQ).$  The converse inclusion comes  form the fact that $A$ is idempotent as follows. Take any $a\in A$ and write $a = \sum _i a_i a'_i.$ Then for any $ q\in Q$ we have $qa = \sum _i q a_i a'_i = \sum _i f_i (qa_i) = \sum _i a_i f_i (q),  $ where $f_i = \cdot a'_i.$
  	
  To check injectivity, note that, if $a\in l(A)$ then $_BQ a=\, _B Q Aa=0, $ showing that $l(A)$ is contained in the  kernel of the canonical map. Conversely,  if $a$ belongs to the kernel, then $Qa=0$, so that, $Aa=\pr{P}{Q}a=\pr{P}{Qa}=0$ and  $a\in l(A)$. This means that $A/l(A)\flecha \END(_BQ)$ is injective and hence $A/l(A)\flecha A\cdot \END(_BQ)$ is an isomorphism.
  
  	(2) This immediately follows from Remark~\ref{rem:item(2)NoSiFaltaSinTorcion}. 
  	
  	(3) By (2) it is enough to prove that $Q\otimes_A -/t_A(Q\otimes_A -)$ is graded isomorphic to $B\cdot\HOM_A(P,-)$. By (3) of Proposition~\ref{primeras propiedades del contexto}, we have the bimodule epimorphism  $\psi : Q \to  B\cdot\HOM_A(P,A),$ whose kernel is $t_B(Q).$ Then for any $U\in A\gr$ we have the surjective map 
    $\psi \otimes 1_U : Q \otimes _A U \to   B\cdot\HOM_A(P,A) \otimes U,$ and composing it with  the epimorphism $\eta $ from   Lemma~\ref{gamma con el producto}, we obtain the surjective map
    $$ Q \otimes _A U \to B\cdot\HOM_A(P,A) \otimes _A U \to B\cdot\HOM_A(P,U) . $$ 

     Write   $\delta=\eta  \circ (\psi\otimes 1_U ),
    $ and we shall see that $\delta$ has torsion kernel. 
     Suppose that $\delta(\sum q\otimes u)=0$. Then, by the definition of $\delta$, we have that, $0=\sum \psi(q)(p)u=\sum\pr{p}{q}u$ for all $p\in P$. Now, take any $b\in B$, and write $b=\sum_i[q'_i,p_i],$ $q'_i \in Q,$ $p_i\in P$. Then 
    \begin{eqnarray*}
   b \sum q\otimes u &=&\sum\sum_i[q'_i,p_i]q\otimes  u=\sum\sum_i q'_i\pr{p_i}{q}\otimes _A u=
   \\&& \sum_i q'_i \otimes \sum \pr{p_i}{q}u=\sum_i q'_i \otimes 0=0.
    \end{eqnarray*}
    This means the kernel of $\delta$ is torsion and, consequently,  $\ker\,\delta=t_B(Q\otimes _A U).$ Hence $(Q \otimes _A U)/t_B( Q\otimes _A U)\cong B\cdot\HOM_A(P,U)$, and we are done.
  \end{proof}

\section{Applications}\label{sec:Applications}

 Suppose that  $A$ and  $B$ are idempotent torsion-free graded rings and let $F:A\gr\flechag B\gr$ and  $G:B\gr\flechag A\gr$ be inverse graded category equivalences. Fix a natural isomorphism $\eta:G \circ F\flecha 1_{A\gr}.$
For $M \in A\gr$ denote by $\mathcal{L}_A^g(M)$ the lattice of graded left unital $A$-submodules of $M.$
 Same way as in \cite[Proposition 21.7]{AF} we obtain the following fact.
 
 \begin{proposition}\label{prop:ModuleLatticeIso} With the above notation,  there is a lattice isomorphism
 \[\Lambda_M:\mathcal{L}_A^g(M)\flecha \mathcal{L}_B^g(F(M))\]
 such that, for any graded unital submodule $K\leq M,$ if we denote by $\imath _K:K\flecha M$ the natural inclusion, then $\Lambda_M(K)=Im\,(F(\imath _K))\leq F(M)$. Moreover, for any graded unital submodule $N\leq F(M),$ the inverse lattice isomorphism takes $N$ to  $Im\,(\eta \circ G(\imath _N)), $ where $\imath _N : N \to F(M)$ is the inclusion map.
 \end{proposition} 
 
Clearly, if $S\in A\gr$ is a graded simple module then  for any non-zero graded submodule $N\leq S$ we have that $AN\neq 0$ is a graded unital submodule. From this fact and  Proposition~\ref{prop:ModuleLatticeIso}, we obtain the following consequence.

 \begin{corollary}\label{cor:SimpleGoToSimple}
	 Suppose that  $A$ and  $B$ are idempotent torsion-free graded rings and let $F:A\gr\flechag B\gr$ and  $G:B\gr\flechag A\gr$ be inverse graded category equivalences. If $S\in A\gr$ is graded simple (or semisimple) then $F(S)$ is graded simple (or semisimple)  in $B\gr$. 
\end{corollary}
\begin{proof}
 The assertion on graded simplicity  has been explained already, and  for semisimple modules we may use that fact that $F$ distributes over the direct sums.
\end{proof}
 
 We proceed  by giving a classical relationship between the graded two-sided ideals of $A$ and $B$. We denote by 
 $\mathcal T_{A}$ the lattice of all graded two-sided ideals $I$ of $A$ which are unital, i.e. $IA = AI =I. $ Obviously $I$ is a torsion-free $A$-bimodule, as so is $A.$ Define similarly  $\mathcal{T}_B$. 

 \begin{proposition}\label{prop:IdealLatticeIso}
  Let  $A$ and $B$ be graded  equivalent torsion-free idempotent graded rings with the associated  graded Morita context $(A,B,P,Q,\langle -,- \rangle, [-,-])$ given by Theorem~\ref{Car Equiv}.  Then, there exists a lattice isomorphism between $\mathcal{T}_A$ and $\mathcal{T}_B$, given by $I\mapsto [QI,P]$ and $J\mapsto \pr{PJ}{Q}$.
 \end{proposition}
 \begin{proof}
  First note that, for any $I\in\mathcal{T}_A$ and $J\in\mathcal{T}_B$ we have that $[QI,P]\in \mathcal{T}_B$ and 
  $\langle PJ,Q \rangle \in \mathcal{T}_A.$
  Then

  \[I \mapsto  [QI,P] \mapsto \pr{P[QI,P]}{Q}=\pr{\pr{P}{QI}P}{Q} =\pr{P}{Q} I\pr{P}{Q}=I, \] and similarly, $J \mapsto \pr{PJ}{Q} \mapsto [Q \pr{PJ}{Q},P] = J, $
proving our result.
 \end{proof}

 We shall prove that being faithful is an invariant property under graded equivalence.  Clearly, a left $A$-module is faithful if and only if for any two-sided ideal $I\leq A$ the equality  $IM=0$ implies $I=0$.
  
 \begin{lemma}\label{lem:faithful}
 Let  $A$ and $B$ be graded equivalent torsion-free idempotent graded rings with the associated  graded Morita context $(A,B,P,Q,\langle -,- \rangle, [-,-])$ given by Theorem~\ref{Car Equiv}. If $M\in A\gr$ is faithful then $B\cdot \HOM(P,M)$ is also faithful.
\end{lemma}
\begin{proof}
Suppose that $M\in A\gr$ is not faithful. Then, taking $I=Ann_A(M)\neq 0$ we have that $IM=0$ and then $AIM= 0.$ Since $A$ is torsion-free, we also have  $AI\neq 0$ and   $J=AIA \neq 0$. Then obviously, $J$ is two-sided unital and, since $J\subseteq A,$ it is torsion-free, so that $J\in \mathcal T _A.$  Write $F=B\cdot \HOM(P,-)$.

By Proposition~\ref{hom y tor} and Proposition~\ref{isomorfismo clave} we know that there is a surjective graded map of graded $A$-modules $\varphi:P\otimes_B F(M)\flecha M$ with torsion kernel.

Now, 
\begin{eqnarray*}
P\otimes_B [QJ,P]F(M)&=& P [QJ,P] \otimes_B F(M) = \pr{P}{QJ}P  \otimes_B F(M) \\&=& \pr{P}{Q}JP  \otimes_B F(M) = AJP  \otimes_B F(M) =   JP\otimes_B F(M)
\end{eqnarray*}
and then  $$\varphi(P\otimes_B [QJ,P]F(M))=\varphi(JP\otimes_B F(M))=J\varphi(P\otimes_B F(M))=JM=0,$$ which implies that $P\otimes_B [QI,P]F(M)=t_B(P\otimes_B [QI,P]F(M))$.

Since $(P\otimes_B-)/t_B((P\otimes_B-)$ is an equivalence, we conclude that $ [QJ,P]F(M)=0.$ Thanks to the lattice isomorphism in Proposition~\ref{prop:IdealLatticeIso}, $ [QJ,P]\neq 0 ,$   so that $F(M)$ cannot be faithful.
\end{proof}

   Finally, the definition of a graded equivalence, Proposition~\ref{prop:ModuleLatticeIso}, Corollary~\ref{cor:SimpleGoToSimple}, Proposition~\ref{prop:IdealLatticeIso} and Lemma~\ref{lem:faithful} imply the next fact.

 \begin{proposition} For idempotent torsion-free graded rings, the following properties are invariant under graded equivalences:  being graded simple, being graded left (or right) semisimple,  being graded left (or right) primitive, having   the graded Jacobson  radical to be zero (i.e. being graded semiprimitive). 
  \end{proposition}

  \section*{Acknowledgements}
The first-named author was partially supported by 
Funda\c c\~ao de Amparo \`a Pesquisa do Estado de S\~ao Paulo (Fapesp), process no.  2020/16594-0, and by  Conselho Nacional de Desenvolvimento Cient\'{\i}fico e Tecnol{\'o}gico (CNPq), process no. 312683/2021-9. 
Both authots were partially supported by the Spanish Government Grant PID2024-155576NB-I00 funded by  MICIU/AEI/ 10.13039/501100011033 /FEDER, UE and Fundaci{\'o}n S{\'e}neca (22004/PI/22). The first-named  author is grateful to the Department of Mathematics of the University of Murcia for its cordial  hospitality during his
 visit.  The second-named author would  like to thank the Department of Mathematics of the University of S{\~a}o Paulo  for its warm hospitality during his visits.


\begin{thebibliography}{10}
\bibitem{Abrams} G.\ D.\  Abrams, Morita equivalence for rings with local units, {\it Commun. Algebra}
\textbf {11} (1983), 801--837.
\bibitem{AF} F. W. Anderson and K. R. Fuller, \textit{Rings and Categories of Modules}, Springer, Berlin, 1974.
\bibitem{Boisen} P. Boisen, Graded Morita Theory, {\it J.~Algebra}
     \textbf {164} (1994), 1--25.
\bibitem{AnhMarki} P.\ N.\ \'{A}nh, L.\ M\'{a}rki, Morita equivalence
for rings without identity,  Tsukuba J. Math. {\bf 11}
(1987), 1--16.
\bibitem{AbDoEx2024} F.\ Abadie, M.\ Dokuchaev, R.\ Exel, 
Strong equivalence of graded algebras,
{\it J. Algebra},  {\bf 659}, (2024),  818--858.
 \bibitem{GS} J. L. Garc\'{i}a and J. J. Sim\'on, Morita Equivalence for Idempotent Rings, {\it J. Pure Appl. Alg.}, {\bf 76}, (1991), 39--56. 
 \bibitem{GordonGreen} R.\ Gordon,  E.\ Green,  Graded Artin algebras, {\it J. Algebra}  \textbf {76}  (1982), 111--137.
\bibitem{Haefner94} J. Haefner,
Graded Morita theory for infinite groups,
{\it J. Algebra}, \textbf { 169}, (1994),   no. 2, 552--586.
\bibitem {Haefner95} J. Haefner, Graded equivalence theory with applications,
{\it J. Algebra} \textbf {172},  (1995), no. 2,   385--424.
  \bibitem{hazrat}  R. Hazrat, {\it Graded Rings and Graded Grothendieck Groups}, London Mathematical Society Lecture Note Series, vol. 435, Cambridge University Press, Cambridge (2016).
  \bibitem{MeniniNast} C.\ Menini and C.\ N\u {a}st\u {a}sescu, When is R-gr equivalent to the category of
modules?, {\it J. Pure  Appl. Algebra} \textbf {51} (1988), 277--291.
  \bibitem{methods} C. Nastasescu.  F. Van Oystaeyen, \textit{Methods of Graded Rings}, Springer, Berlin, 2004.
  \bibitem{Nobusawa} N. Nobusawa, $\Gamma$-rings and Morita equivalence of rings, {\it Math. J. Okayama Univ.} \textbf{26}, (1984), 151--156.
 \end{thebibliography}
\end{document}